\documentclass[11pt]{article}
\usepackage{amssymb,amsmath,bm}
\usepackage[colorlinks=true, urlcolor=blue,linkcolor=blue, citecolor=blue]{hyperref}
\usepackage{mathrsfs,verbatim}
\usepackage{float}
\usepackage{caption}
\usepackage{constants,color}
\usepackage{enumerate}
\usepackage{bbm}
\usepackage{appendix}
\usepackage{tikz}
\usepackage{tkz-tab}
\usepackage{tikz-qtree}

\usetikzlibrary{arrows}
\usetikzlibrary{positioning}
\newcommand{\E}{\mathbb{E}}
\newcommand{\Pro}{\mathbb{P}}
\newcommand{\Var}{\mathrm{Var}}
\newcommand{\Cov}{\mathrm{Cov}}

\newcommand{\1}{\mathbf{1}}

\begin{document}

\newcommand{\bea}{\begin{eqnarray}}
\newcommand{\ena}{\end{eqnarray}}
\newcommand{\beas}{\begin{eqnarray*}}
\newcommand{\enas}{\end{eqnarray*}}
\newcommand{\beq}{\begin{equation}}
\newcommand{\enq}{\end{equation}}
\def\qed{\hfill \mbox{\rule{0.5em}{0.5em}}}
\newcommand{\bbox}{\hfill $\Box$}
\newcommand{\ignore}[1]{}
\newcommand{\ignorex}[1]{#1}
\newcommand{\wtilde}[1]{\widetilde{#1}}
\newcommand{\mq}[1]{\mbox{#1}\quad}
\newcommand{\bs}[1]{\boldsymbol{#1}}
\newcommand{\qmq}[1]{\quad\mbox{#1}\quad}
\newcommand{\qm}[1]{\quad\mbox{#1}}
\newcommand{\nn}{\nonumber}
\newcommand{\Bvert}{\left\vert\vphantom{\frac{1}{1}}\right.}
\newcommand{\To}{\rightarrow}
\newcommand{\supp}{\mbox{supp}}
\newcommand{\law}{{\cal L}}
\newcommand{\Z}{\mathbb{Z}}
%Erd\"{o}s-R\'{e}nyi

\newtheorem{theorem}{Theorem}[section]
\newtheorem{corollary}{Corollary}[section]
\newtheorem{conjecture}{Conjecture}[section]
\newtheorem{proposition}{Proposition}[section]
\newtheorem{lemma}{Lemma}[section]
\newtheorem{definition}{Definition}[section]
\newtheorem{example}{Example}[section]
\newtheorem{remark}{Remark}[section]
\newtheorem{case}{Case}[section]
\newtheorem{condition}{Condition}[section]
\newcommand{\pf}{\noindent {\it Proof:} }
\newcommand{\proof}{\noindent {\it Proof:} }
\frenchspacing

\title{\bf Some remarks on biased recursive trees}
\author{Ella Hiesmayr\footnote{Ko\c{c} University, Istanbul, Turkey. email: ehiesmayr17@ku.edu.tr} \hspace{0.2in} \"{U}m\.{i}t I\c{s}lak\footnote{Bo\u{g}azi\c{c}i University, Istanbul, Turkey. email: umit.islak1@boun.edu.tr}} \vspace{0.25in}

\maketitle

\begin{abstract}
The purpose of this paper is to analyze certain statistics of a recently introduced non-uniform random tree model, biased recursive trees. This model is based on constructing a random tree by establishing a correspondence with  non-uniform permutations, biased riffle shuffles. The statistics that are treated include the number of nodes with a given number of descendants, the depth of the tree, and the number of branches.  The model  yields the uniform recursive trees as a certain limit,  some new results for the uniform  case  are  obtained as well.

\bigskip


\bigskip

\end{abstract}

\section{Introduction}\label{sec:intro}

A \emph{uniform recursive tree},  abbreviated by  URT, of order $n$ is a random tree that is chosen uniformly among all $(n-1)!$ increasing trees with  $n$ nodes. 
This is equivalent to the following recursive construction principle. A URT $\mathcal{T}_n$ of order $n$ is obtained from a URT $\mathcal{T}_{n-1}$ with  $n-1$ nodes  by joining node $n$ to any of the nodes $\{1, \dots, n-1\}$ with equal probability. If we start from scratch, first node $1$ is added as the root, and node $2$ is attached to node $1$. Then, node $3$ is either attached to node $1$ or to node $2$  with equal probability $\frac{1}{2}$. In general node $i$ attaches to any of the nodes $\{1, \dots, i-1\}$ with probability $\frac{1}{i-1}$ \cite{Altokislak}.

Another  way to construct a URT that is useful in understanding certain tree statistics is based on a bijection between recursive trees and permutations. For example, in the uniformly random case,  one  generates a uniformly random
permutation (URP) on $\{2,\ldots,n\}$ and a  URT on vertices $\{1,\ldots,n\}$
incrementally at the same time,  allowing one to have enough independence to handle random tree problems via well known results on functions of independent random variables. The construction will be  explained in detail in the following section. 

There is a vast literature on URTs. We refer to \cite{Mahmoud} for a survey on some classical results on several statistics of URTs including the height, extremal degrees and internodal distances. Applications of URTs are also rich, see for example,  \cite{Feng05},    \cite{Gastwirth},  \cite{Fire}, \cite{NajockHeyde82} 
for different approaches in modeling  real life problems. Despite the applications we have for URTs, it lacks significant properties: the chance to have a different global recursive construction principle and the chance to have nodes with distinct behaviors. For this reason, various non-uniform recursive tree models have been introduced in the last two decades. 

Variations of recursive trees in non-uniform setting include Hoppe trees \cite{Hoppe}, weighted recursive trees \cite{HI:2017}, the binary recursive trees  \cite{Flajolet}, plane-oriented recursive trees \cite{Szymanski87} and scaled attachment random recursive trees \cite{Devroye11}.    Our interest here is in yet another   model which was recently introduced in \cite{Altokislak}. Here, the construction of the random tree  is based on a completely different approach using  interplay between the symmetric group and recursive trees. In comparison to the  URT case, we use  a biased riffle shuffle distribution instead of a URP distribution on the symmetric group and then consider the  corresponding trees obtained from these permutations.This allows us  to derive results on the number of branches, the number of nodes with at least $k$ descendants, the number of nodes with exactly $k$ descendants and the depth of nodes in this so called biased recursive tree (BRT) model. Also, as to be seen below, URTs turn out to be a certain limiting case of BRTs , and we exploit this to analyze the number of nodes with exactly $k$, and with at least $k$ descendants in URTs.

The rest of the paper is organized as follows. In Section \ref{sec:connectionstoperms}, we review riffle shuffles and  discuss the construction of biased recursive trees via riffle shuffles. Later in Section \ref{sec:branches}  we  present our results on the number of branches. This is followed by an analysis of the  number of nodes with a given number of descendants in BRTs and URTs in Sections  \ref{sec:descendants} and \ref{sec:URT}, respectively. The last statistic of interest, the depth of node $n$ is investigated in Section \ref{sec:depth} and then the paper 
is concluded   with a brief discussion in Section \ref{sec:conclusion}.  
To avoid cumbersome notation we will take $ a < \{b_1, \dots, b_k \}$ to mean that $a< b_i$ for all $1 \leq i \leq k$ throughout the paper.

\section{Connections to random permutations}\label{sec:connectionstoperms}

\subsection{URT constructions based on permutations}

We begin by reviewing constructions of URTs via uniformly random permutations. The discussion here closely follows \cite{Altokislak}. In order to construct  the URT, we first construct a uniformly random permutation  $\pi$ of $\{2,3,\dots, n\}$.
Given this $\pi$, we construct a URT as follows: First, $2$ is attached to $1$, then $3$ is connected to $1$ if it is to the left of $2$ in the permutation, otherwise to $2$. In general node $i$ is attached to the rightmost node to the left of $i$ that is less than $i$. If there is no smaller number than $i$ to its left, $i$ is attached to 1. Thus $i= \pi(s)$ attaches to node $j=\pi(r)$ if $r = \max \{t \in \{1, \dots, s-1\}: \pi(t) < \pi(s) \}$, where we set $\pi(1) = 1$. We will call this way of constructing a URT the \emph{construction from a permutation}. Figure \ref{fig:StepByStep} shows an example of the step by step construction of a recursive tree corresponding to a permutation. 

\begin{figure}
  \tikzstyle{every node}=[circle, draw]
\begin{tikzpicture}[>=stealth',auto,node distance=1.6cm,
  thick, font=\sffamily, main node/.style={circle,draw,font=\sffamily\bfseries}]
  
  \node[draw=none,rectangle]  (p) {\underline{1}6387254};
  
  \node[main node] (1) [below=3mm  of p] {1};

\begin{scope}[xshift=2cm]
         \node[draw=none,rectangle]  (p)  {\textbf{1}6387\underline{2}54};
  \node[main node] (1) [below=3mm of p] {1};
  \node[main node] (2)  [below of=1] {2};
  
\path[every node/.style={font=\sffamily\small}]
    (1) edge node {} (2);

\end{scope}

\begin{scope}[xshift=4.5cm]     

          \node[draw=none,rectangle]  (p) {\textbf{1}6\underline{3}87254};
  \node[main node] (1) [below= 3mm of p] {1};
  \node[main node] (2)  [below left of=1] {2};
  \node[main node] (3) [below right of =1] {3};
  
\path[every node/.style={font=\sffamily\small}]
    (1) edge node {} (2)
    (1) edge node {} (3);

\end{scope}

\begin{scope}[xshift = 8cm]
 \node[draw=none,rectangle]  (p) {16378\textbf{2}5\underline{4}};
   \node[main node] (1) [below=3mm of p] {1};
  \node[main node] (2)  [below left of=1] {2};
  \node[main node] (3) [below right of =1] {3};
  \node[main node] (4) [below of =2] {4};
  
\path[every node/.style={font=\sffamily\small}]
    (1) edge node {} (2)
    (1) edge node {} (3)
    (2) edge node {} (4);
    \end{scope}
    
    \begin{scope}[xshift = 12.5cm]
    \node[draw=none,rectangle]  (p)  {16378\textbf{2}\underline{5}4};
   \node[main node] (1) [below=3mm of p] {1};
  \node[main node] (2)  [below left of=1] {2};
  \node[main node] (3) [below right of =1] {3};
  \node[main node] (4) [below left of =2] {4};
  \node[main node] (5) [below right of=2] {5};
  
\path[every node/.style={font=\sffamily\small}]
    (1) edge node {} (2)
    (1) edge node {} (3)
    (2) edge node {} (4)
    (2) edge node {} (5);
    \end{scope}
    
        \begin{scope}[xshift = 2cm, yshift = -5cm]
          \node[draw=none,rectangle]  (p)  {\textbf{1}\underline{6}387254};
   \node[main node] (1) [below=3mm of p] {1};
  \node[main node] (2)  [below left of=1] {2};
  \node[main node] (3) [below of =1] {3};
  \node[main node] (4) [below left of =2] {4};
  \node[main node] (5) [below of=2] {5};
  \node[main node] (6) [below right of =1] {6};
  
\path[every node/.style={font=\sffamily\small}]
    (1) edge node {} (2)
    (1) edge node {} (3)
    (2) edge node {} (4)
    (2) edge node {} (5)
    (1) edge node {} (6);
    \end{scope}
    
            \begin{scope}[xshift = 7cm, yshift = -5cm]
             \node[draw=none,rectangle]  (p)  {16\textbf{3}8\underline{7}254};
   \node[main node] (1) [below= 3mm of p] {1};
  \node[main node] (2)  [below left of=1] {2};
  \node[main node] (3) [below of =1] {3};
  \node[main node] (4) [below left of =2] {4};
  \node[main node] (5) [below of=2] {5};
  \node[main node] (6) [below right of =1] {6};
  \node[main node] (7) [below of =3] {7};
  
\path[every node/.style={font=\sffamily\small}]
    (1) edge node {} (2)
    (1) edge node {} (3)
    (2) edge node {} (4)
    (2) edge node {} (5)
    (1) edge node {} (6)
    (3) edge node {} (7);
    \end{scope}
    
\begin{scope}[xshift = 12cm, yshift = -5cm]
  \node[draw=none,rectangle]  (p)  {16\textbf{3}\underline{8}7254};
   \node[main node] (1) [below=3mm of p] {1};
  \node[main node] (2)  [below left of=1] {2};
  \node[main node] (3) [below of =1] {3};
  \node[main node] (4) [below left of =2] {4};
  \node[main node] (5) [below of=2] {5};
  \node[main node] (6) [below right of =1] {6};
  \node[main node] (7) [below of =3] {7};
  \node[main node] (8) [below right of =3] {8};
  
\path[every node/.style={font=\sffamily\small}]
    (1) edge node {} (2)
    (1) edge node {} (3)
    (2) edge node {} (4)
    (2) edge node {} (5)
    (1) edge node {} (6)
    (3) edge node {} (7)
    (3) edge node {} (8);
    \end{scope}

\end{tikzpicture}
\caption{Step by step construction of the recursive tree correponding to 16387254. The newly attached node is underlined and its parent bold.}   \label{fig:StepByStep}
\end{figure}
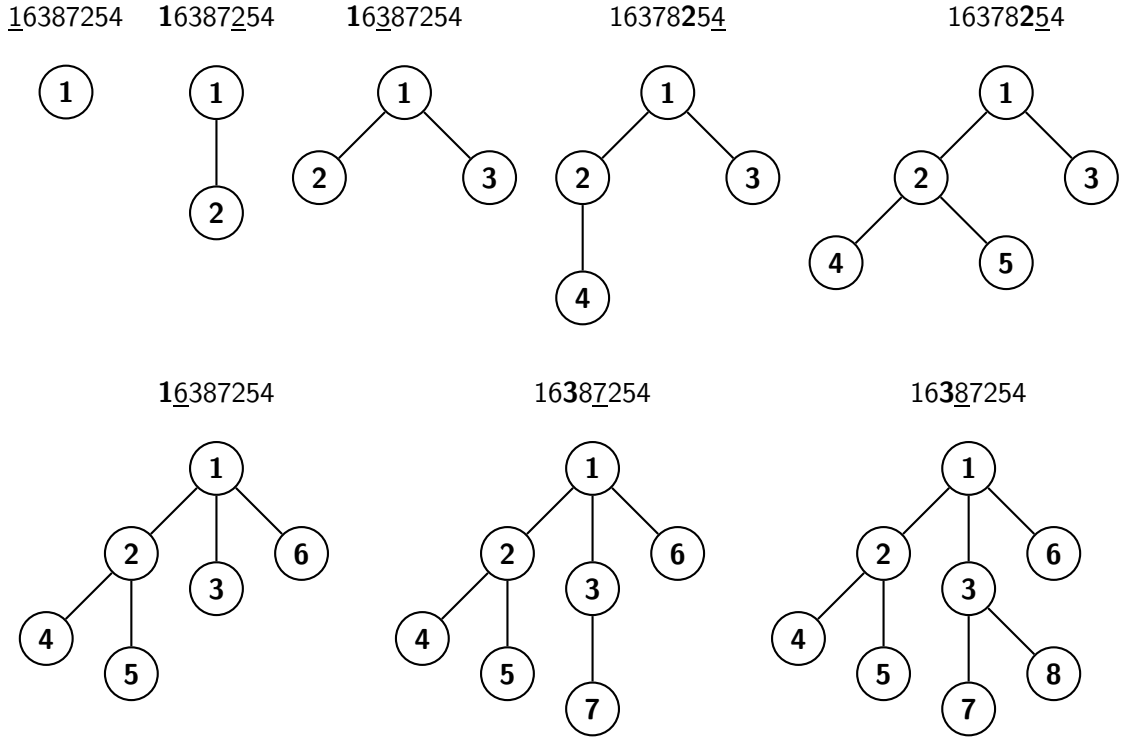

Similarly, given a URT $\mathcal{T}_n$, we can construct the corresponding permutation by writing $1$ to the very left, $2$ on its right and then step by step every node $i$ directly on the right of the node it is attached to. That this relation gives a bijection between URTs and uniform random permutations can be seen by the symmetry of the recursive constructions. The tree and the permutation corresponding to it can also be constructed simultaneously, see \cite{Altokislak}.

\subsection{Riffle shuffles and BRTs}

First we note that the construction of a URT via permutations described in the previous section can be formalized as:  

\bigskip

Let $\pi=\pi_2 \ldots \pi_{n}$ be a  permutation on
$\{2,\ldots,n\}$, then we can construct a recursive tree with vertices
$1,2,\ldots,n$ by taking the vertex $1$ as the root, and attaching the
node $i \geq 2$ to the rightmost element $j$ of $\pi$, which
precedes $i$ and is less than $i$. If there is no such element $j$,
define the root 1 to be the parent of $i$. 
If we start with a uniform random permutation, then the resulting tree is a URT.
\bigskip

Our treatment below will be based on this interpretation. In particular we will replace the uniformly random permutation with a random permutation that is sampled from a general riffle shuffle distribution. 
Details of the following brief review of riffle shuffles can be found in \cite{Diac01} and \cite{Fulman}.

\begin{definition}  Cut the $n \geq 2$ card deck into $a$ piles by picking pile sizes
  according to the $mult(a;p = (p_1,\ldots,p_a))$ distribution.
That is, choose $b_1,\ldots,b_a$ with probability
$\binom{n}{b_1,\ldots,b_a} \Pi_{i=1}^a p_i^{b_i}.$ Then choose
uniformly one of  the $\binom{n}{b_1,\ldots,b_a}$ ways of interleaving
the packets, leaving the cards in each pile in their original order.
The resulting probability distribution on $S_n$ is called the
\emph{$p$-biased riffle shuffle distribution} and is denoted by
$P_{n,p}$. When $p$ is the uniform distribution over $[a]:\{1,2,\ldots,a\}$ for some $a \in \mathbb{N}$, we write  $P_{n,a}$ to denote the resulting distribution and call
the resulting shuffle  an \emph{$a$-shuffle}.
\end{definition}

The following equivalence will turn out to be crucial for our
arguments.
\begin{proposition} (Inverse shuffles) The inverse of a $p$-biased shuffle with $a$ piles has
the following description. Assign independent random digits from
  $[a]$ to each card with distribution $p=(p_1,\ldots,p_a)$. Then sort according to digit,
  preserving relative order for cards with the same digit.
\end{proposition}

Using $p$-biased  shuffles we next introduce $p$-biased recursive trees.
Let $\gamma^p=(\gamma_1,\ldots,\gamma_{n-1})$ have the $p$-biased
shuffle distribution over  $\{2,\ldots,n\}$. Then construct a recursive
tree with vertices $1,2,\ldots,n$ by taking the vertex $1$ as the root,
and attaching the node $i \geq 2$ to the rightmost element $j$ of
$\gamma^p$, which precedes $i$ and is less than $i$. If there is no
such an element $j$, then we define the root 1 to be the parent of
$i$. Call the resulting random tree $\mathcal{T}_n'$.
\begin{definition}
The random tree $\mathcal{T}_n'$ described in  previous paragraph
will be called a $p$\emph{-biased recursive tree ($p$-BRT)}. When
$p$ is the uniform distribution over $[a]$, $a\geq 2$, we call
$\mathcal{T}_n'$ an $a$\emph{-recursive tree} ($a$-RT).
\end{definition}

The following figure shows all 2-RTs one may obtain on $4$ vertices.

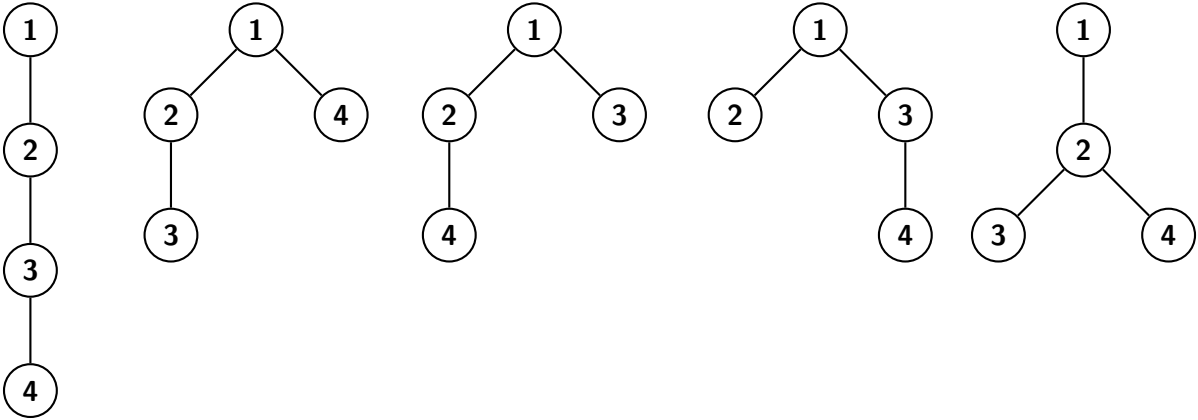
\begin{figure}[H]
\begin{center}
\begin{tikzpicture}[>=stealth',auto,node distance=1.6cm,
  thick,main node/.style={circle,draw,font=\sffamily\bfseries}]

%\centering
  \node[main node] (1) {1};
  \node[main node] (2) [below of=1] {2};
  \node[main node] (3) [below of=2] {3};
  \node[main node] (4) [below of=3] {4};

    \path[every node/.style={font=\sffamily\small}]
    (1) edge node {} (2)
    (2) edge node {} (3)
    (3) edge node {} (4);

    \node at (0,-7) {};

   \begin{scope}[xshift=3cm]

   \node[main node] (1) {1};
  \node[main node] (2) [below left of=1] {2};
  \node[main node] (3) [below right of=1] {4};
  \node[main node] (4) [below of=2] {3};

\path[every node/.style={font=\sffamily\small}]
    (1) edge node {} (2)
    (1) edge node {} (3)
    (2) edge node {} (4);

    \node at (0,-7) {};

    \end{scope}

   \begin{scope}[xshift=6.7cm]

   \node[main node] (1) {1};
  \node[main node] (2) [below left of=1] {2};
  \node[main node] (3) [below right of=1] {3};
  \node[main node] (4) [below of=2] {4};

\path[every node/.style={font=\sffamily\small}]
    (1) edge node {} (2)
    (1) edge node {} (3)
    (2) edge node {} (4);

    \node at (0,-7) {};

    \end{scope}

    \begin{scope}[xshift=10.5cm]

    \node[main node] (1) {1};
  \node[main node] (2) [below left of=1] {2};
  \node[main node] (3) [below right of=1] {3};
  \node[main node] (4) [below of=3] {4};

\path[every node/.style={font=\sffamily\small}]
    (1) edge node {} (2)
    (1) edge node {} (3)
    (3) edge node {} (4);

    \node at (0,-7) {};

    \end{scope}

        \begin{scope}[xshift=14cm]

      \node[main node] (1) {1};
  \node[main node] (2) [below of=1] {2};
  \node[main node] (3) [below left of=2] {3};
  \node[main node] (4) [below right of=2] {4};

\path[every node/.style={font=\sffamily\small}]
    (1) edge node {} (2)
    (2) edge node {} (3)
    (2) edge node {} (4);

    \end{scope}

   \end{tikzpicture}
\end{center}
\vspace{-1in} \caption{$P_{n,2}$ assigns probability $1/2$ to the leftmost tree, and  $1/8$ to all others.}
\end{figure}

Before moving to our results, let us note some reasons for choosing the biased riffle shuffle distribution among all possible non-uniform distributions on the symmetric group. First of all the equivalent description of riffle shuffles via inverse riffle shuffles   makes the model more tractable thanks to underlying independence.  Secondly, biased riffle shuffle distributions vary a lot depending on the chosen parameters. Thus the properties of the corresponding trees will also be close or far from    the uniform case accordingly.  As an example, the number of branches can be bounded in BRTs if appropriate parameters are chosen in contrast to URTs. This may  allow one to model more diverse recursive phenomena with more precision.  Lastly, riffle shuffle permutations are themselves interesting from both a theoretical and applicational perspective, and have connections to various fields  \cite{Diac01}. 

We may now move on to studying the BRT statistics beginning with the number of branches. In the following we will always assume that $p= (p_1, \dots, p_a)$ is non-degenerate i.e. that every $p_i>0$.  This is not a restriction on the number of different p-BRT models we can consider since simply deleting the zero entry in $p$ will give us the same distribution on the permutations, with $p$ of length $a-1$.

\section{Number of branches}\label{sec:branches}

The \emph{number of branches} in a tree  is the degree of the root. In the analysis of URTs there are several options one may use  to obtain results on the number of branches. See, for example,  \cite{Feng05} for one elementary approach and \cite{ellatez} for a relevant discussion. We will use an approach based on anti-records.

A record is an element that is greater than all previous ones and an anti-record an element that is smaller than all previous ones. More precisely: a permutation $\pi$ has a \emph{record} in $i$ if $\pi(i) > \max \{\pi(1), \dots, \pi(i-1)\}$ and an \emph{anti-record} in $i$ if $\pi(i) < \min \{\pi(1), \dots, \pi(i-1)\}$ \cite{Stanley}. %[p.30]
Every permutation has a record and an anti-record in $\pi(1)$. We can similarly define records for sequences of continuous random variables. Nevzorov summarizes many results about records of sequences of random variables in \cite{Nevzorov01}. These can be extended to uniform random permutations by the standard construction of a uniform random permutation from independent uniformly random variables, for a description see \cite{Altokislak}.

To get results on the number of branches in a BRT we will use the observation that the number of branches of a recursive tree is equal to the number of anti-records in its permutation representation, where we consider the permutation to start with $\pi(2)$. Let us make the connection between the number of anti-records in a permutation $\pi$ and the number of branches in the corresponding tree clear. We can observe that $\pi(2)$ attaches to 1 and is thus the first node of a branch of $\mathcal{T}_n$. If $\pi(3) > \pi(2)$ it will attach to $\pi(2)$, if $\pi(4) > \pi(2)$ it will attach to $\pi(2)$ or $\pi(3)$ and so on. As long as no $\pi(r) < \pi(2)$ appears, all nodes will be part of the branch starting with node $\pi(2)$. Let ${r= \min \{ r = 3, \dots, n : \pi(r) < \pi(2) \}} $ then $\pi(r)$ is the second node that attaches to 1 and is thus the start of the next branch. Similarly, as long as no smaller number comes up, all subsequent nodes will be part of this second branch. In general, $\pi(r)$ attaches to 1 if and only if ${\pi(r) < \min \{\pi(2), \dots, \pi(r-1)\}}$ which means that $\pi(r)$ must be an anti-record. Thus, we can also write the number of branches $\mathcal{B}_n$ of $\mathcal{T}_n$ as another sum of indicators. 

\begin{proposition}
Let  $\mathcal{B}_{n}^{p}$ be the number of branches in  a $p$-BRT  and $\pi \in S_n$ be a URP. Then 
$$\mathcal{B}_n(\mathcal{T}_n^p) =_d  \#(\text{anti-records in} \; \pi) =\sum_{r=2}^{n} \1({\pi(r) = \min \{\pi(1), \dots, \pi(r)\}}),$$ where $=_d$ is used for equality in distribution.

\end{proposition}

Now, using anti-records our goal is to find explicit values for  the first two moments of the number of branches in a BRT.  We begin with  the expectation.

\begin{theorem}
Let $\mathcal{B}_{n}^{p}$ denote the number of branches in a $p$-BRT $\mathcal{T}_{n}^{p}$. Then 
\begin{equation*} \label{thm:BranchesExpBRT}
\begin{split}
\E[\mathcal{B}_{n}^{p}] = \sum_{s=1}^{a-1} \frac{p_s}{\sum_{\ell=1}^{s} p_{\ell}} \left ( 1- \left (\sum_{\ell=s+1}^{a} p_{\ell} \right)^{n-1} \right ) + p_a.
\end{split}
\end{equation*}
Moreover,  
\begin{equation*}
\E[\mathcal{B}_{n}^{p}]  \xrightarrow{n \to \infty} \sum_{s=1}^{a} \frac{p_s}{\sum_{\ell=1}^{s} p_{\ell}}.
\end{equation*}
When $p$ is the uniform distribution over $[a]$, letting $\mathcal{B}_n^{a}$ be the number of branches in an $a$-RT $\mathcal{T}_n^{a}$,  
\begin{equation*}
\E[\mathcal{B}_{n}^{a}]= \sum_{s=1}^{a-1} \frac{1}{s} \left ( 1- \left ( 1- \frac{s}{a}\right )^{n-1} \right )+ \frac{1}{a}.
\end{equation*}
Asymptotically we have for fixed $a$
\begin{equation*}
\E[\mathcal{B}_{n}^{a}] \xrightarrow{n \to \infty} H_a,
\end{equation*}
and if we fix $n$ and let  $a$ tend  to infinity, then the expectation tends to the expectation for the URT case, namely
\begin{equation*}
\E[\mathcal{B}_{n}^{a}] \xrightarrow{a \to \infty} H_{n-1},
\end{equation*}
where $H_n$ is the $n^{th}$ harmonic number of the first order, i.e. $H_n = \sum_{i=1}^{n} \frac{1}{n}$.
\end{theorem}

\begin{proof}
We will use the inverse formulation for biased riffle shuffles in order to study the number of anti-records. As we consider riffle shuffles of $\{2, \dots, n\}$, there definitely is an anti-record at $\gamma(2)$. For $i>2$ we get an anti-record at $i$ if and only if $X_i < \min\{X_2, \dots, X_{i-1}\}$, since only in that case we get at $i$ the first card from a pile containing cards that are smaller than all the previous ones.
Since the $X_i$'s are i.i.d. and $\Pro(X_j = s)= p_s$ for $s \in [a]$, we get for $3\leq i < n$,
\begin{eqnarray*}
 \Pro (X_i < \min\{X_2, \dots, X_{i-1}\})  &=& \sum_{s=1}^{a} \Pro(X_i < \min\{X_2, \dots, X_{i-1}\} | X_i = s) \Pro(X_i=s)   \\
&=& \sum_{s=1}^{a-1} \left ( \sum_{\ell=s+1}^{a} p_{\ell}\right )^{i-2} p_s.
\end{eqnarray*}
We sum from 1 to $a-1$ because for $i\geq3$, $\min\{X_2, \dots, X_{i-1}\} \leq a$ and thus $\Pro(a < \min\{X_2, \dots, X_{i-1}\}) = 0$.
Let us denote $A_i =  \1(\gamma(i) \text{ is an anti-record})$, then $\mathcal{B}_n^p = \sum_{i=2}^{n} A_i$ and $A_2=1$. Hence,
\begin{eqnarray*}
\E[\mathcal{B}_n^p]
= 1+  \sum_{i=3}^n  \Pro (X_i < \min\{X_2, \dots, X_{i-1}\}) &=& 1 + \sum_{i=3}^n \sum_{s=1}^{a-1} \left ( \sum_{\ell=s+1}^{a} p_{\ell}\right )^{i-2} p_s \\
&=& \sum_{s=1}^{a-1} \frac{p_s}{\sum_{\ell=1}^{s} p_{\ell}} \left ( 1- \left (\sum_{\ell=s+1}^{a} p_{\ell} \right)^{n-1} \right ) + p_a.
\end{eqnarray*}
Since  $p_1>0$ implies  that $\left (\sum_{\ell=s+1}^{a} p_{\ell} \right)^{n-1} \xrightarrow{n \to \infty} 0$ for all $s\geq 1$, we get, \begin{equation*}\E[\mathcal{B}_n^p]  \xrightarrow{n \to \infty}{} \sum_{s=1}^{a-1} \frac{p_s}{\sum_{\ell=1}^{s} p_{\ell}} + p_a = \sum_{s=1}^{a} \frac{p_s}{\sum_{\ell=1}^{s} p_{\ell}}.\end{equation*}
In particular, for $p$ the uniform distribution over $[a]$ we have $p_i= \frac{1}{a}$, so we get
\begin{equation} \label{ExpectBranchA}
\E[\mathcal{B}_n^a]
= \sum_{s=1}^{a} \frac{1}{s} \left ( 1- \left ( 1- \frac{s}{a}\right )^{n-1} \right ) 
\end{equation}
and 
\begin{equation*}\E[\mathcal{B}_n^a]  \xrightarrow{n \to \infty}  \sum_{s=1}^{a} \frac{1}{s} = H_a.\end{equation*}

To show that for  fixed $n$ fixed, if $a \to \infty$, the expectation tends to the expectation of the number of branches for URTs we will manipulate the right-most term  of \eqref{ExpectBranchA} to get a formulation involving a geometric sum and then use the following identity ${\sum_{s=1}^{a} s^k = \frac{a^{k+1}}{k+1} + \mathcal{O}(a^k)}$, that holds for all $k \in \mathbb{N}_0$. We have
\begin{eqnarray*}
\E[\mathcal{B}_n^a] &=& \sum_{s=1}^{a-1} \frac{1}{s} \left ( 1- \left ( \frac{a-s}{a}\right )^{n-1} \right )+ \frac{1}{a} 
=  \sum_{s=1}^{a-1} \frac{1}{a} \frac{ 1- \left ( \frac{a-s}{a}\right )^{n-1} }{1-\frac{a-s}{a}}+ \frac{1}{a} \\
&=&  \sum_{\ell=0}^{n-2} \frac{1}{a^{\ell+1}}  \left [ \frac{a^{\ell+1}}{\ell+1} + \mathcal{O}(a^{\ell}) \right ]+ \frac{1}{a}
= \sum_{\ell=0}^{n-2} \frac{1}{\ell+1} + \mathcal{O}\left( \frac{1}{a} \right) \xrightarrow{a \to \infty} H_{n-1},
\end{eqnarray*}
concluding the proof.
\hfill $\square$
\end{proof}

\begin{remark}
The observation that $\mathbb{E}[\mathcal{B}_n^a] \rightarrow H_{n-1}$ as $a \rightarrow \infty$ is no surprise. $H_{n-1}$ is the same as the expected number of branches in a URT on $n$ vertices and since $X_i$'s tend to be distinct when $a$ is large, $a$-RTs get close to URTs. We will see this asymptotically similar  behaviour throughout the paper and in particular use it in Section \ref{sec:URT} to derive results on URTs from our results about BRTs.
\end{remark}

\begin{remark} That the number of branches is equal to the number of anti-records allows us to conclude immediately that a $p$-BRT can have at most $a$ branches because for $\gamma(i)$ to be an anti-record in a riffle shuffle permutation, it must be the first card of one of the $a$ piles. 
Note that in general, if one does not require an explicit value for the expected value, then it is easy to obtain upper bounds for the expected value of the number of branches, also for infinite $a$. A simple approach would be to consider a Markov chain on $\mathbb{N}$, that starts at time $t=1$ with distribution $p$ (possibly with countable support), and which moves along $\mathbb{N}$ by choosing the new state according to $p$ independently of all else. Then clearly the number of branches will be less or equal to the hitting time of the chain to state $1$. Since the expected hitting time in this case is merely $1/p_1$, we obtain $\mathbb{E}[\mathcal{B}_n^p]  \leq 1 /p_1$. 
\end{remark}

\begin{remark}
The previous remark in particular implies that for a given $p$, $\mathbb{E}[\mathcal{B}_n^p] $ is uniformly bounded in $n$.  A  BRT type random tree with unbounded  branches can be formed if one may replace riffle shuffles by unfair permutations  \cite{AIP}, \cite{Prodinger}. 
This alternative random  permutation model is formed via independent and continuous $X_i$'s removing the chance of them being equal and allowing an infinite number of anti-records, because the $X_i$s are not distributed on a finite set.
\end{remark}

Next, we focus on the variance of the number of branches. 
\begin{theorem} \label{thm:BRTVarBranches}
Let $\mathcal{B}_{n}^{p}$ denote the number of branches of a $p$-BRT $\mathcal{T}_n^p$. Then
\begin{eqnarray*}  \label{equ:VarBranches}
\Var(\mathcal{B}_n^p) &=& 
 \sum_{s=1}^{a-1} \frac{p_s}{\sum_{\ell=1}^{s}p_{\ell}} \left ( 1- \left ( \sum_{\ell=s+1}^{a} p_{\ell}\right ) ^{n-1}\right )   - \sum_{s=1}^{a-1} p_s 
 - \sum_{s=1}^{a-1} p_s^2  \left (\sum_{\ell=s+1}^{a} p_{\ell} \right )^2  \frac{1- \left (\sum_{\ell=s+1}^{a} p_{\ell} \right )^{2(n-2)}}{1 -\left (\sum_{\ell=s+1}^{a} p_{\ell} \right )^{2}} \\
&-& 2  \sum_{s=2}^{a-1} \sum_{r=1}^{s-1} p_s p_r \left (\sum_{\ell=s+1}^{a} p_{\ell} \sum_{\ell=r+1}^{a} p_{\ell} \right ) \frac{1- \left ( \sum_{\ell=s+1}^{a} p_{\ell} \sum_{\ell=r+1}^{a} p_{\ell} \right )^{n-2}}{1 - \sum_{\ell=s+1}^{a} p_{\ell} \sum_{\ell=r+1}^{a} p_{\ell} }\\
&+& 2  \sum_{s=2}^{a-1} \frac{p_s \sum_{\ell=s+1}^{a}p_{\ell}}{\sum_{\ell=1}^{s}p_{\ell}} \sum_{r=1}^{s-1} \frac{p_r}{\sum_{q = 1}^{r} p_q}  \left ( 1 -  \left ( \sum_{\ell=s+1}^{a}p_{\ell}\right )^{n-3}\right )\\
&-& 2 \sum_{s=2}^{a-1}p_s \sum_{\ell=s+1}^{a}p_{\ell} \sum_{r=1}^{s-1} \frac{p_r \sum_{q = r+1}^{a} p_q}{\sum_{q = 1}^{r} p_q}\frac{1}{\sum_{q = r+1}^{s} p_q} \cdot \left ( \left (\sum_{q = r+1}^{a} p_q\right)^{n-3} - \left( \sum_{\ell=s+1}^{a}p_{\ell}\right)^{n-3} \right ) \\
&-& 2 \sum_{s=1}^{a-1}p_s  \sum_{\ell=s+1}^{a}p_{\ell} \sum_{r=1}^{a-1}\frac{p_r}{\sum_{\ell=1}^{r}p_{\ell} }  \left ( \sum_{\ell=r+1}^{a}p_{\ell} \right )^2 \frac{1- \left ( \sum_{\ell=s+1}^{a}p_{\ell} \sum_{\ell=r+1}^{a}p_{\ell}\right )^{n-3}}{1- \sum_{\ell=s+1}^{a}p_{\ell} \sum_{\ell=r+1}^{a}p_{\ell}} \\
&+& 2 \sum_{s=1}^{a-1}\frac{p_s}{\sum_{\ell=1}^{s}p_{\ell}}  \sum_{\ell=s+1}^{a} p_{\ell} \sum_{r=1}^{a-1}\frac{p_r}{\sum_{\ell=1}^{r}p_{\ell} } \left ( \sum_{\ell=r+1}^{a}p_{\ell}\right )^{n-1} \left ( 1 - \left ( \sum_{\ell=s+1}^{a}p_{\ell}\right )^{n-3}\right )  
\end{eqnarray*}
and
\begin{eqnarray*}
\Var(\mathcal{B}_n^p) &\xrightarrow{n \to \infty}& 
  \sum_{s=1}^{a-1} \frac{p_s}{\sum_{\ell=1}^{s}p_{\ell}}  - \sum_{s=1}^{a-1} p_s  
-  \sum_{s=1}^{a-1} p_s^2  \left (\sum_{\ell=s+1}^{a} p_{\ell} \right )^2  \frac{1}{1 -\left (\sum_{\ell=s+1}^{a} p_{\ell} \right )^{2}} \\
&& - 2  \sum_{s=2}^{a-1} \sum_{r=1}^{s-1} p_s p_r \left (\sum_{\ell=s+1}^{a} p_{\ell} \sum_{\ell=r+1}^{a} p_{\ell} \right ) \frac{1}{1 - \sum_{\ell=s+1}^{a} p_{\ell} \sum_{\ell=r+1}^{a} p_{\ell} }\\
&& + 2  \sum_{s=2}^{a-1} \frac{p_s \sum_{\ell=s+1}^{a}p_{\ell}}{\sum_{\ell=1}^{s}p_{\ell}} \sum_{r=1}^{s-1} \frac{p_r}{\sum_{q = 1}^{r} p_q}  \\
&& - 2 \sum_{s=1}^{a-1}p_s  \sum_{\ell=s+1}^{a}p_{\ell} \sum_{r=1}^{a-1}\frac{p_r}{\sum_{\ell=1}^{r}p_{\ell} }  \left ( \sum_{\ell=r+1}^{a}p_{\ell} \right )^2 \frac{1}{1- \sum_{\ell=s+1}^{a}p_{\ell} \sum_{\ell=r+1}^{a}p_{\ell}}. \\
\end{eqnarray*}

\end{theorem}

\begin{proof} Let $\gamma$ be the permutation representation of  $\mathcal{T}_n^{p}$.
For the calculation of the variance difficulty arises from the dependence of the events $A_i=\1(\gamma(i) \text{ is an anti-record})$.  We will express the variance as 
$
\Var(\mathcal{B}_n^p) = \Var\left ( 1+ \sum_{i=3}^n A_i \right ) = \sum_{i=3}^{n} \Var(A_i) + 2 \sum_{3 \leq i < j \leq n} \Cov (A_i,A_j).$
For $3 \leq i < j \leq n$
\begin{eqnarray*}
 \E[A_i A_j]  &=& \Pro \left (X_i < \min\{X_2, \dots, X_{i-1}\}, X_j < \min\{X_2, \dots, X_{j-1}\} \right ) \\
&=& \Pro \left (X_j < \min\{X_2, \dots, X_{j-1}\} | X_i < \min\{X_2, \dots, X_{i-1}\} \right ) 
 \cdot \Pro \left (X_i < \min\{X_2, \dots, X_{i-1}\}\right )\\
&=& \sum_{s=1}^{a-1} \Pro \left (X_j < \min_{k=2, \dots, j-1}\{X_k\} | X_i < \min_{k=2, \dots, i-1}\{X_k\}, X_i=s \right )   \\
&& \quad  \cdot \; \;  \Pro \left (X_i < \min_{k=2, \dots, i-1}\{X_k\} | X_i =s \right ) \Pro \left (X_i =s \right ).
\end{eqnarray*}
If $X_i=1$, $X_j < X_i$ is not possible, so the above expression is only positive for $2\leq s \leq a-1$. For these $s$ we have:
\begin{equation*}
\begin{split}
& \Pro \left (X_j < \min_{k=2, \dots, j-1}\{X_k\} | X_i < \min_{k=2, \dots, i-1}\{X_k\}, X_i=s \right ) \\
&= \sum_{r=1}^{s-1} \underbrace{\Pro \left ( X_j < \min_{k=2,\dots, j-1}\{X_k\} | X_j = r, X_i =s , X_i< \min_{k=2, \dots, i-1}\{X_k\} \right )}_{= \Pro(r < X_{i+1}, \dots, X_{j-1})} \Pro(X_j=r)\\
&= \sum_{r=1}^{s-1}\left (\sum_{q = r+1}^{a} p_q\right)^{j-i-1}p_r.
\end{split}
\end{equation*}
So in total $3 \leq i < j \leq n$ we have for ,
\begin{equation*}\E[A_iA_j] = \sum_{s=2}^{a-1}p_s \left ( \sum_{\ell=s+1}^{a}p_{\ell}\right )^{i-2}  \sum_{r=1}^{s-1} p_r \left (\sum_{q = r+1}^{a} p_q\right)^{j-i-1} .
\end{equation*}
By summing over all $i$ and $j$ and some straightforward but lengthy simplifications we then get the above results. For details see \cite{ellatez}. \hfill $\square$
\end{proof}

\begin{theorem} \label{thm:aRTBranchesVariance}
Let $\mathcal{B}_n^{a}$ denote the number of branches in an $a$-RT. Then
\begin{eqnarray*}\label{VarBranchesA}
\Var(\mathcal{B}_n^a) &=& \hspace{1em} \sum_{s=1}^{a-1} \frac{1}{s} \left ( 1- \left ( \frac{a-s}{a}\right ) ^{n-1}\right )  - \frac{a-1}{a} 
 - \sum_{s=1}^{a-1} \frac{1}{a^2}  \left ( \frac{a-s}{a} \right )^2  \frac{1- \left (\frac{a-s}{a} \right )^{2(n-2)}}{1 -\left (\frac{a-s}{a} \right )^{2}} \\
&& - 2  \sum_{s=2}^{a-1} \sum_{r=1}^{s-1} \frac{1}{a^2} \frac{a-s}{a} \frac{a-r}{a} \frac{1- \left ( \frac{a-s}{a} \frac{a-r}{a} \right )^{n-2}}{1 - \frac{a-s}{a}  \frac{a-r}{a} }
 +2  \sum_{s=2}^{a-1}\frac{a-s}{as} \sum_{r=1}^{s-1} \frac{1}{r}  \left ( 1 -  \left (\frac{a-s}{a}\right )^{n-3}\right )\\
&&  - 2 \sum_{s=2}^{a-1} \frac{a-s}{a^2} \sum_{r=1}^{s-1} \frac{a-r}{r}\frac{1}{ s-r} \left ( \left (\frac{a-r}{a}\right)^{n-3} - \left(\frac{a-s}{a}\right)^{n-3} \right ) \\
&&  - 2 \sum_{s=1}^{a-1}\frac{a-s}{a^2} \sum_{r=1}^{a-1}\frac{1}{r}  \left (\frac{a-r}{a} \right )^2 \frac{1- \left ( \frac{a-s}{a} \frac{a-r}{a}\right )^{n-3}}{1- \frac{a-s}{a} \frac{a-r}{a}} \\
&& + 2 \sum_{s=1}^{a-1}  \frac{a-s}{sa} \sum_{r=1}^{a-1}\frac{1}{r} \left (\frac{a-r}{a}\right )^{n-1} \left ( 1 - \left (\frac{a-s}{a}\right )^{n-3}\right),
\end{eqnarray*}
and 
\begin{eqnarray*} \label{VarBranchesANinfty}
\Var(\mathcal{B}_n^{a}) \xrightarrow {n \to \infty}   H_a  - \frac{a-1}{a}  + \frac{2}{a} \sum_{s=2}^{a-1}\frac{a-s}{s} \sum_{r=1}^{s-1} \frac{1}{r}
-  \frac{1}{a^2} \sum_{s=1}^{a-1}  \frac{s^2}{a^2 - s^{2}} 
&-&  \frac{2}{a^2}   \sum_{s=1}^{a-2} \sum_{r=1}^{a-s+1} \frac{sr}{a^2 - sr }\\
&-& \frac{2}{a^2} \sum_{s=1}^{a-1} \sum_{r=1}^{a-1}\frac{1}{r} \frac{s (a-r)^2}{a^2- s (a-r)}. 
\end{eqnarray*}
Moreover, for fixed $n$,  when $a$ increases the variance approaches the variance from the uniform case:
\begin{equation}\Var(\mathcal{B}_n^{a}) \xrightarrow {a \to \infty} H_{n-1}- H_{n-1}^{(2)}. \end{equation}
\end{theorem}

The proof of this last statement involves generating functions and some easy asymptotic arguments. We do not include it the proof here since it is too lengthy. It  can be found in \cite{ellatez}. In general, throughout the paper, we do not include the variance computations except the proofs of Theorem \ref{thm:kDesVarGeneral} and Corollary \ref{cor:aRTkDescendantsVar}, and we refer to the first author's thesis for these calculations.

\begin{remark}
Another approach to the number of branches of a BRT is obtained by constructing a riffle shuffle permutation by sequential shuffling, see \cite{Bayer92}. The first card from pile $i$ can only be an anti-record if it comes before any card of a pile with a smaller index. By considering that we can first shuffle the first $i-1$ piles and then shuffle the obtained shuffled pile with pile $i$, this means that the first card from pile $i$ must be on top after this shuffling step. The probability of this event is in turn proportional to the pile sizes. Conditioning on the pile sizes of the riffle shuffle permutation, these are independent events and we can thus obtain the expected number of branches and a formula for the variance. For details see \cite{ellatez}.
\end{remark}

\begin{remark}
We argued at the beginning that the number of branches is limited by the number of piles we split the deck into, i.e. $a$. That argument  also works for the degree of any node: if two nodes are children of $i$, the smaller one must come later in the permutation, thus the nodes attached to $i$ form a decreasing sequence. But in a $p$-biased riffle shuffle permutation any decreasing sequence must consist of cards from different piles since cards from the same pile remain in the same order. Since there are $a$ piles, this implies that the degree of any node is limited by $a$. Since $a$-ary trees also have this property, this suggests the question whether $a$-ary recursive trees share other properties with trees constructed from $a$-shuffles, or $p$-biased riffle shuffle with $|p|=a$.
\end{remark}

\section{Number of nodes with  exactly $k$ descendants}\label{sec:descendants}
   
A \emph{descendant} of any vertex $v$ is any vertex which is either the child of $v$ or is the descendant of any of the children of $v$.     
In a BRT the number of nodes with at least $k$ descendants can also be analyzed using the construction of an inverse riffle shuffle. The number of nodes with exactly $k$ descendants in binary recursive trees was previously studied by Devroye \cite{Devroye}.  In particular, denoting the number of nodes with exactly $k$ left descendants in a uniform binary tree of size $n$ by $L_{kn}$, he  shows that $$\frac{L_{kn} - n p_k }{\sqrt{n} \sigma_k} \longrightarrow_d N(0,1),$$ as  $n \rightarrow \infty$, where $$\sigma_k^2 = p_k (1 - p_k)  - 2 (k + 1) p_k^2  + 2 \rho_k,$$ and $$p_k =  \frac{1}{(k+2)(k+1)} \qquad \text{and} \qquad \rho_k = \frac{1}{(2 k + 3) (2 k + 2) (k  + 1)}.$$  Devroye also notes that the number of nodes with $k$ descendants in a uniform random recursive tree is equal to the number of nodes in the
associated random binary search tree with $k$ left descendants. For a special case, recalling that a node with no descendants is called a \emph{leaf}, the number of leaves in URTs were previously studied in \cite{Altokislak}, \cite{Zhang}, \cite{NajockHeyde82}, \cite{NaRapoport} and \cite{Flajolet} among others.

Turning our attention back to BRTs, a node $i$ has at least $k$ descendants if in the permutation at least $k$ entries after $i$ are larger than $i$. This is the case if, given $\gamma(j)=i$, we have ${\gamma(j+1), \dots, \gamma(j+k) > i}$, which is equivalent to $X_j \leq \{ X_{j+1}, \dots, X_{j+k} \}$ in the inverse riffle shuffle construction. We thus get a node with at least $k$ descendants for every $X_j$ satisfying the above. Here is an example. 

\begin{example}
Let our deck consist of 8 cards. Assume we cut it into 3 piles. We first construct the inverse riffle shuffle by assigning digits from 1 to 3 to every card as can be seen in Figure \ref{fig:RiffleShuffleInverseConstruction}.
\begin{figure}[H]
\begin{center}
\begin{tabular}{ c c c c c c c }
$ X_2$ & $X_3$ & $X_4$ & $X_5$ & $X_6$ & $X_7$ & $X_8  $\\ 
1 & 2 & 3 & 1 & 3 & 1  & 1 
\end{tabular}
\end{center}
\caption{Example of the construction of an inverse riffle shuffle permutation.} \label{fig:RiffleShuffleInverseConstruction}
\end{figure}
This gives the inverse permutation $\gamma^{-1} = 2578346$ and thus $\gamma=2673845$ with the corresponding recursive tree in Figure \ref{fig:RifSufExDes}. 

\begin{figure}[H]
\tikzstyle{every node}=[circle, draw]

\begin{center}
\begin{tikzpicture}[node distance=1.6cm,
  thick,main node/.style={circle,draw,font=\sffamily\bfseries}]
  \node[main node] (1) {1};
  \node[main node] (2) [below of=1] {2};
  \node[main node] (3) [below left of=2] {3};
  \node[main node] (4) [below left of=3] {4};
  \node[main node] (5) [below  of = 4] {5};
  \node[main node] (6) [below right of=2] {6};
  \node[main node] (7) [below  of=6] {7};
  \node[main node] (8) [below right of=3] {8};
 
\path[every node/.style={font=\sffamily\small}]
    (1) edge node {} (2)
    (2) edge node {} (3)
    (2) edge node {} (6)
    (3) edge node {} (4)
    (3) edge node {} (8)
    (4) edge node {} (5)
    (6) edge node {} (7);

\end{tikzpicture}
\end{center}
\caption{The biased recursive tree corresponding to $\pi= 2673845$.} \label{fig:RifSufExDes}
\end{figure}
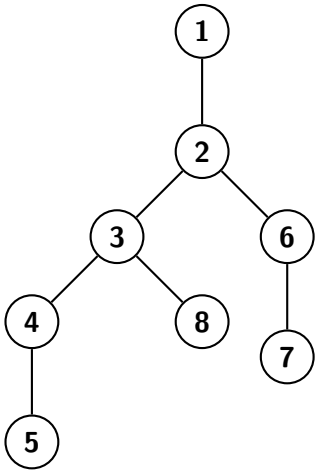
As  can be seen in the figure above, there are 2 nodes with at least 2 descendants, 2 and 3. This corresponds to what we can derive from the permutation: $\pi(5) = 3$ and we have $X_5 \leq \{X_6, X_7\}$. Also $\pi(2)= 2$ and we have $X_2 \leq \{ X_3, X_4\}$.
\end{example}

We will now use this observation to calculate the expectation and variance of the number of nodes with at least $k$ descendants, and will then prove a central limit theorem and a strong law  by making use of the underlying local dependence when $k$ is fixed.

\begin{theorem} \label{thm:BRTExpkDesce}
Let $Y_{\geq k, n}^{p}$ denote the number of nodes with at least $k$ descendants in a $p$-BRT, then
\begin{equation*}
\E[Y_{\geq k, n}^{p}] = (n-k-1)\sum_{s=1}^{a} p_s \left (\sum_{r=s}^{a} p_r  \right )^k +1.
\end{equation*}
Moreover, if $p$ is the uniform distribution over $[a]$, 
\begin{equation*}
\E[Y_{\geq k, n}^{a}] = (n-k-1)  \frac{1}{a^{k+1}} \sum_{s=1}^{a} s^k +1 
\end{equation*}
and in particular, as $a \to \infty$, we get 
\begin{equation*}
\E[Y_{\geq k, n}^{a}] \longrightarrow \frac{n}{k+1}.
\end{equation*}
\end{theorem}

\begin{proof}
Let $C_i^k = \1(X_i \leq \{X_{i+1}, \dots, X_{i+k}\})$, with the $X_j$s having distribution $p$ over $[a]$.. Then
\[
Y_{\geq k, n}^{p} =_d \sum_{i=1}^{n-k} C_i^k = \sum_{i=2}^{n-k} C_i^k +1.
\]
Also for $2 \leq i \leq n-k$, 
\begin{eqnarray*}
\E[C_i^k] = \Pro(X_i \leq \{X_{i+1}, \dots, X_{i+k}\}) 
&=& \sum_{s=1}^{a} \Pro(s \leq \{X_{i+1}, \dots, X_{i+k}\}) \Pro(X_i = s)\\
&=& \sum_{s=1}^{a} p_s \left (\sum_{r=s}^{a} p_r  \right )^k. 
\end{eqnarray*}
Thus 
$$
\E[Y_{\geq k, n}^{p}] = \sum_{i=1}^{n-k} \E[C_i^k]
 = \sum_{i=2}^{n-k} \sum_{s=1}^{a} p_s \left (\sum_{r=s}^{a} p_r  \right )^k +1 
= (n-k-1)\sum_{s=1}^{a} p_s \left (\sum_{r=s}^{a} p_r  \right )^k +1.
$$
In particular, if $p_s = \frac{1}{a}$ for all $s$, we get 
$$
\E[Y_{\geq k, n}^{p}] = (n-k-1) \sum_{s=1}^{a} \frac{1}{a} \left ( \sum_{r=s}^{a} \frac{1}{a} \right )^k +1 
= (n-k-1)  \frac{1}{a^{k+1}} \sum_{s=1}^{a} s^k +1. 
$$
This gives, as $a \longrightarrow \infty$, $$
\lim_{a \to \infty} \E[Y_{\geq k, n}^{p}] = \lim_{a \to \infty} (n-k-1)  \frac{1}{a^{k+1}} \sum_{s=1}^{a} s^k +1
 = \lim_{a \to \infty} (n-k-1)  \frac{1}{a^{k+1}} \left [ \frac{a^{k+1}}{k+1} + \mathcal{O}(a^k) \right ] +1
 =  \frac{n}{k+1}.
$$
\end{proof}

\begin{theorem} \label{thm:kDesVarGeneral}
Let $Y_{\geq k, n}^{p}$ denote the number of nodes with at least $k$ descendants in a $p$-BRT, then
\begin{eqnarray*}
\Var(Y_{\geq k, n}^{p}) &=& \sum_{s=1}^{a} p_s \left (\sum_{r=s}^{a} p_r  \right )^k \left [(n-k-1)+ p_1\left(2nk -3k(k+1)\right) \right ]  \\
&& +2 \sum_{s=2}^{a} p_s  \frac{1}{\sum_{u=1}^{s-1} p_u} \sum_{r=s}^{a} p_r \left ( \sum_{t=r}^{a} p_t \right )^k \\ 
&& \qquad \cdot \left [n-k-1  -(n-2k-1)\left (\sum_{u=s}^{a} p_u\right )^k - \frac{1- \left ( \sum_{u=s}^{a} p_u\right )^{k}}{1-\sum_{u=s}^{a} p_u}  \right ]\\
&& -  \left ( \sum_{s=1}^{a} p_s \left (\sum_{r=s}^{a} p_r  \right )^k \right )^2 \left [ n(2k+1) - (3k+1)(k+1)\right ].
\end{eqnarray*}
We moreover have
\begin{eqnarray*}
\lim_{n \to \infty} \frac{\Var(Y_{\geq k, n}^{p})}{n} =&  \sum_{s=1}^{a} p_s \left (\sum_{r=s}^{a} p_r  \right )^k \left ( 2kp_1 + 1 -  (2k+1) \sum_{s=1}^{a} p_s \left (\sum_{r=s}^{a} p_r  \right )^k \right )  \\
&+2 \sum_{s=2}^{a} \frac{p_s}{\sum_{u=1}^{s-1} p_u} \sum_{r=s}^{a} p_r \left ( \sum_{t=r}^{a} p_t \right )^k \left [1 -\left (\sum_{u=s}^{a} p_u\right )^k \right ].
\end{eqnarray*}
\end{theorem}
For the proof of the theorem and its following corollary, see the Appendix. 

\begin{corollary} \label{cor:aRTkDescendantsVar}
If we choose the uniform distribution over $[a]$ we get
\begin{eqnarray*}  \label{eq:VarkDesART}
\Var(Y_{\geq k , n}^{a}) &=&\frac{1}{a^{k+1}}\sum_{s=1}^{a} s^k\left [(n-k-1)+ \frac{1}{a}\left(2nk -3k(k+1)\right) \right ] 
+2 \frac{1}{a^{k+1}} \sum_{s=2}^{a} \frac{1}{s-1} \sum_{r=1}^{a-s+1} r^k \\
&& \qquad \cdot \left [n-k-1  -(n-2k-1)\left ( \frac{a-s+1}{a}\right )^k - \frac{1- \left ( \frac{a-s+1}{a}\right )^{k}}{1-\frac{a-s+1}{a}}  \right ]\\
&& -  \left (\frac{1}{a^{k+1}}\sum_{s=1}^{a} s^k \right )^2 \left [ n(2k+1) - (3k+1)(k+1)\right ],
\end{eqnarray*}
and for fixed $a$ we moreover have
\begin{equation*}
\lim_{n \to \infty} \frac{\Var(Y_{\geq k , n}^{a})}{n} =
\frac{1}{a^{k+1}}\sum_{s=1}^{a} s^k\left [1 + \frac{2k}{a} - \frac{2k+1}{a^{k+1}}\sum_{s=1}^{a} s^k \right ] 
 + \frac{2}{a^{k+1}} \sum_{s=1}^{a-1} \frac{1}{s} \left [1  - \left ( \frac{a-s}{a}\right )^k  \right ] \sum_{r=1}^{a-s+1} r^k .
\end{equation*}
If we fix $n$ we get 
\begin{equation*}
\begin{split}
\Var(Y_{\geq k , n}^{a}) \xrightarrow{a \to \infty}&  \frac{n-k-1}{k+1}-  2 \frac{(n-k-1)}{k+1} H_{k+1}  + 2 \frac{n-2k-1}{k+1} H_{2k+1} \\
& +  \frac{2}{k+1} \sum_{\ell=0}^{k-1} H_{k+\ell+1} - \frac{ n(2k+1) - (3k+1)(k+1)}{(k+1)^2},
\end{split}
\end{equation*}
and in particular 
\begin{equation*}
\lim_{n\to \infty} \lim_{a \to \infty} \frac{\Var(Y_{\geq k , n}^{a})}{n} =   \frac{1}{k+1} + \frac{2H_{2k+1}-2H_{k+1}}{k+1} - \frac{2k+1}{(k+1)^2}.
\end{equation*}
\end{corollary}

The next result shows that $Y_{\geq k,n}$ satisfies a central limit theorem for fixed  $k$.

\begin{theorem}
Let  $Y_{\geq k, n}^{p}$ denote the number of nodes with at least $k$ descendants in a $p$-BRT. Then 
\begin{equation*}
d_W\left  (Y_{\geq k, n}^{p}, \mathcal{G} \right )  \leq  \frac{2k+1}{\sigma} \left ( (2k+1) + \frac{\sqrt{28}(2k+1)^{\frac{1}{2}}}{\sqrt{\pi}} \right ),
\end{equation*}
where $d_W$ is the Wasserstein distance between probability measures, and $\mathcal{G}$ is a standard normal random variable. 
\end{theorem}
\begin{proof}
We will  use Theorem   3.6 of \cite{Ross}: Let $Y_1, Y_2, \dots, Y_n$ be random variables for which $\E[Y_i^4] < \infty$ and $\E[X_i]=0$ holds. Let moreover $N_i$ be the dependency neighbourhoods of $(Y_1, \dots, Y_n)$ and define $D:= \max_{1\leq i \leq n} \{|N_i|\}$. Finally set $\sigma^2 = \Var\left ( \sum_{i=1}^n Y_i\right )$ and define $W:= \sum_{i=1}^n \frac{Y_i}{\sigma}$. Then
\[d_W(W,\mathcal{G}) \leq \frac{D^2}{\sigma^3} \sum_{i=1}^{n} \E[|Y_i|^3] + \frac{\sqrt{28}D^{\frac{3}{2}}}{\sqrt{\pi}\sigma^2} \sqrt{\sum_{i=1}^n \E[Y_i^4]}. \] Here,  For each $i$, we call $N_i$ the \emph{dependency neighbourhood} of $Y_i$, if $Y_i$ is independent of $\{Y_j\}_{j \notin N_i}$ and $i \in N_i$.

Now, going back to our problem,    we set $Y_i := C_i^{k}-\E[C_i^{k}]$ for $i=1, \dots, n-k$. The dependency neighbourhoods are $N_i = \{C_j^k: i-k \leq j \leq i+k\}$ for $i\leq n-2k$ and $N_i = \{C_j^k: i-k \leq j \leq n-k\}$ for $i= n-2k+1, \dots, n-k$. Hence we have $D= 2k+1$. Let $\sigma^2 = \Var \left ( \sum_{i=1}^n C_i^{k} \right )$ and define $W:= \sum_{i=1}^n \frac{C_i^{k}-\E[C_i^{k}]}{\left (\Var(\sum_{i=1}^n C_i^{k} \right )^{\frac{1}{2}}}$.

We now need to estimate $\sum_{i=1}^{n} \E[|Y_i|^3]$ and $\sum_{i=1}^n \E[Y_i^4]$.
Since the $Y_i$ can take values $1-p_i$ or $-p_i$ where $p_i = \E[C_i^{k}]$, we have  
\begin{equation*}
\E[|Y_i|^3] = |1-p_i|^3 p_i + |-p_i|^3(1-p_i) 
= p_i(1-p_i) ((1-p_i)^2 + p_i^2) 
= p_i(1-p_i) (1-2p_i(1- p_i)) 
\leq \Var[C_i^{k}]
\end{equation*}
since for $0<a<1$, we have $0<a(1-a)<\frac{1}{4}$ and thus $1>1-2p_i(1- p_i)>\frac{1}{2}$.
Similarly 
\begin{equation*}
\E[|Y_i|^3]  = (1-p_i)^4 p_i + p_i^4(1-p_i) 
= p_i(1-p_i) ((1-p_i)^3 + p_i^3) 
= p_i(1-p_i) (1-3(pi(1-p_i)) 
\leq \Var(C_i^{k}).
\end{equation*}
since for $0<a<1$, we have $1>1-3p_i(1- p_i) > \frac{1}{4}$.

Therefore,  we obtain
\begin{equation*}
\begin{split}
d_W(W,\mathcal{G}) & \leq \frac{(2k+1)^2}{\sigma^3} \sum_{i=1}^{n-k} \Var(C_i^k)+ \frac{\sqrt{28}(2k+1)^{\frac{3}{2}}}{\sqrt{\pi}\sigma^2} \sqrt{\sum_{i=1}^n \Var(C_i^k)} \\
&  \leq \frac{2k+1}{\sigma} \left ( (2k+1) + \frac{\sqrt{28}(2k+1)^{\frac{1}{2}}}{\sqrt{\pi}} \right ).
\end{split}
\end{equation*}
As we know by Theorem \ref{thm:kDesVarGeneral} that the variance is of order $n$, this bound decreases with order $\mathcal{O}\left ( \frac{1}{\sqrt{n}} \right ).$
\hfill $\square$
\end{proof}

Note that since $Y_{\leq k, n}^{p}$, the number of nodes with at most $k$ descendants, is equal to $n- Y_{\geq k+1, n}^{p}$,  the expectation, variance and CLT of $Y_{\leq k, n}^{p}$ follows directly from the above results.

\begin{remark}
When $k$ grows moderately with $n$, it is highly likely that  the number of nodes with $k$ descendants will tend to the Poisson distribution since the event that a given node is has at least $k$ leaves becomes a rare event. We do not go into details of this here, noting  that one may obtain convergence via standard Poisson approximation results under local dependence. 
\end{remark}

Our results on $Y_{\geq k,n}^{\cdot}$ can also be used to study the   number of nodes with exactly $k$ descendants.

\begin{corollary} \label{cor:ExactKDesBRTExpect}
Let $X_{k,n}^{p}$ denote the number of nodes with exactly $k$ descendants in a $p$-BRT. Then
\begin{equation*}
\E[X_{k,n}^{p}] = (n-k-1)\sum_{s=1}^{a} p_s \left ( \left (\sum_{r=s}^{a} p_r  \right )^k - \left (\sum_{r=s}^{a} p_r  \right )^{k+1} \right ) + \sum_{s=1}^{a} p_s \left (\sum_{r=s}^{a} p_r  \right )^{k+1}.
\end{equation*}
Moreover if $X_{k,n}^{a}$ is the number of nodes with exactly $k$ descendants in a $a$-RT, then
\begin{equation*}
 \E[X_{k,n}^{a}] = (n-k-1) \left (  \frac{1}{a^{k+1}} \sum_{s=1}^{a} s^k -  \frac{1}{a^{k+2}} \sum_{s=1}^{a} s^{k+1}\right )  + \frac{1}{a^{k+2}} \sum_{s=1}^{a} s^{k+1}, 
\end{equation*}
and asymptotically we have 
\begin{equation*} 
\E[X_{k,n}^{a}] \xrightarrow{a \to \infty} \frac{n}{(k+1)(k+2)}.
\end{equation*}
\end{corollary}
\begin{proof}
We have $X_{k,n}^{p} = Y_{\geq k, n}^{p} - Y_{\geq k+1, n}^{p}$ thus   Theorem \ref{thm:BRTExpkDesce} yields
\begin{equation*}
\begin{split}
& \E \left [X_{k,n}^{p} \right ] = \E\left [Y_{\geq k, n}^{p} \right]- \E \left [Y_{\geq k+1, n}^{p} \right ] \\
&= (n-k-1)\sum_{s=1}^{a} p_s \left (\sum_{r=s}^{a} p_r  \right )^k +1 - (n-k-2)\sum_{s=1}^{a} p_s \left (\sum_{r=s}^{a} p_r  \right )^{k+1} - 1 \\
&= (n-k-1)\sum_{s=1}^{a} p_s \left ( \left (\sum_{r=s}^{a} p_r  \right )^k - \left (\sum_{r=s}^{a} p_r  \right )^{k+1} \right ) + \sum_{s=1}^{a} p_s \left (\sum_{r=s}^{a} p_r  \right )^{k+1}.  \\
\end{split}
\end{equation*}
Similarly, when $p$  is the uniform uniform distribution,   Theorem \ref{thm:BRTExpkDesce}  gives 
\begin{equation*}
\E[X_{k,n}^{a}] = (n-k-1) \left (  \frac{1}{a^{k+1}} \sum_{s=1}^{a} s^k -  \frac{1}{a^{k+2}} \sum_{s=1}^{a} s^{k+1}\right )  + \frac{1}{a^{k+2}} \sum_{s=1}^{a} s^{k+1}. \\
\end{equation*}
This implies in particular that
\begin{eqnarray*}
\E[X_{k,n}^{a}]
&=& (n-k-1) \left (  \frac{1}{a^{k+1}} \left [ \frac{a^{k+1}}{k+1} + \mathcal{O}(a^k) \right ] -  \frac{1}{a^{k+2}} \left [ \frac{a^{k+2}}{k+2} + \mathcal{O}(a^{k+1}) \right ] \right )  \\
&&+ \frac{1}{a^{k+2}} \left [ \frac{a^{k+2}}{k+2} + \mathcal{O}(a^{k+1}) \right ]\\
&\xrightarrow{a \to \infty}&  (n-k-1) \left ( \frac{1}{k+1} -  \frac{1}{k+2} \right ) +  \frac{1}{k+2} 
= \frac{n}{(k+1)(k+2)}.
\end{eqnarray*}
\hfill $\square$
\end{proof}
\begin{corollary}
Let $X_{k,n}^{a}$ denote the number of nodes with exactly $k$ descendants in an $a$-RT, then
\begin{eqnarray}\label{eqn:varxkna}
\nonumber\Var(X_{k,n}^{a} ) &=& 
\frac{1}{a^{k+1}}\sum_{s=1}^{a} s^k \left [(n-k-1) - \frac{2}{a} \left( n-k-1\right ) \right ]  +\frac{2}{a^{k+1}} \sum_{s=2}^{a}  \frac{1}{s-1} \sum_{r=1}^{a-(s-1)} r^k  \\
\nonumber &&\cdot \left [- \frac{n-2k-1}{a^k} \left (a-(s-1)\right )^k  - \frac{1}{a^{k-1}} \frac{a^{k}- a^{k-1}(a-(s-1)) - \left (a-(s-1)\right )^{k}}{s-1} \right ]\\
\nonumber && - \frac{1}{a^{2k+2}} \left ( \sum_{s=1}^{a} s^k \right )^2 \left [ n(2k+1) - (3k+1)(k+1)\right ]   \\
\nonumber && + \frac{1}{a^{k+1}} \sum_{s=1}^{a} s^{k+1} \left [-n+3k+ \frac{1}{a} \left [ 2n-7k-6\right ] \right ] 
 +\frac{2}{a^{k+2}} \sum_{s=2}^{a} \frac{1}{s-1} \sum_{r=1}^{a-(s-1)} r^{k+1} \\
\nonumber && \cdot \left [-\frac{n-2k-3}{a^{k+1}}\left (a-(s-1)\right )^{k+1} - \frac{1}{a^{k}} \frac{a^{k+1}- a^{k}(a-(s-1)) - (a-(s-1) )^{k+1}}{s-1} \right ]\\
&& -  \frac{1}{a^{2k+4}} \left ( \sum_{s=1}^{a} s^{k+1} \right )^2 \left [ n(2k+3) - (3k+4)(k+2)\right ] \\
\nonumber &&+  \frac{2}{a^{2k+2}} \sum_{s=2}^{a} \frac{1}{s-1} \sum_{r=1}^{a-(s-1)}  r^{k+1} \left ( a-(s-1) \right )^k \left [ n-2k-1 -\frac{a}{s-1} \right ] \\
\nonumber &&+ \frac{2}{a^{2k+2}} \sum_{s=2}^{a} \frac{1}{s-1} \sum_{r=1}^{a-(s-1)} r^{k} \left (a-(s-1)\right )^{k+1} \left [ n-2k-1-  \frac{a}{ s-1} \right ]\\
\nonumber && +\frac{2}{a^{2k+3}} \sum_{s=1}^{a} s^k \sum_{r=1}^{a} r^{k+1} \left [ 2(n-2k-1)(k+1) + k^2 \right].
\end{eqnarray}
\end{corollary}

As a final result in this section we note the following strong law. 
\begin{theorem}
We have $$\frac{X_{k,n}^{a}}{n} \longrightarrow \frac{1}{a^{k+1}} \sum_{s=1}^a s^k - \frac{1}{a^{k+2}} \sum_{s=1}^a s^{k+1} \qquad a.s., $$ as $n \rightarrow \infty$. 
\end{theorem}

\begin{proof}
First observe that we are able to express   $X_{k,n}^{a}$ as a function of $n$  independent random variables, and that changing only one of these will do a change of at most $2k$  in the value of  $\mathcal{L}_{k, n}$. Therefore we may use the bounded differences inequality (See, for example, Theorem 6.2 in \cite{BLM})to conclude that 
\begin{equation*}
\mathbb{P}(|X_{k,n}^{a}- \mathbb{E}[X_{k,n}^{a}]| > \sqrt{n} \sqrt{\ln (n^{4k^2})}) \leq 2 \exp \left(- 2 \frac{n \ln (n^{4k^2}) }{4 k^2  n} \right) 
 =  \frac{2}{n^2}.
\end{equation*}
Using Borel-Cantelli's first lemma we may then conclude that  $$|X_{k,n}^{a} - \mathbb{E}[X_{k,n}^{a}]|  = o(n)$$ a.s.. Therefore $\frac{X_{k,n}^{a} - \mathbb{E}[X_{k,n}^{a}]}{n} \rightarrow 0$ a.s., and result follows since $$\frac{\mathbb{E}[X_{k,n}^{a}]}{n} \longrightarrow \frac{1}{a^{k+1}} \sum_{s=1}^a s^k - \frac{1}{a^{k+2}} \sum_{s=1}^a s^{k+1}.$$

\hfill $\square$
\end{proof}
\section{Uniform recursive tree case}\label{sec:URT}

Our purpose in this section is to adapt our analysis from the previous section to uniformly recursive trees. Previously, Devroye  \cite{Devroye},  studied the number of nodes with exactly $k$ left descendants in a uniform binary recursive tree, and proved a central limit theorem for this case. He also notes that via a bijection argument, these results can also be interpreted in terms of URTs, see for instance \cite{Cormen} for a description of the left-child, right-sibling representation of rooted trees.

For  our purposes, we first note the following theorem whose proof follows the same lines as in  \cite{Altokislak}.

\begin{theorem}\label{thm:compare}
Let us denote the number of nodes with at least $k$ descendants in a URT, $p$-BRT and $a$-RT of order $n$ as $\mathcal{L}_{\geq k,n}$, $Y_{\geq k,n}^p$ and $Y_{\geq k,n}^{a}$ respectively. Then
\begin{enumerate}
\item for $n \geq 3$, \[d_{TV}(\mathcal{L}_{\geq k, n},Y_{\geq k,n}^p) \leq \binom{n-1}{2} \sum_{s=1}^{a} p_s^2,\]
\item for $a \geq n\geq 3$  and $p$ the uniform distribution this bound can be improved and we get 
\[d_{TV}(\mathcal{L}_{\geq k, n},Y_{\geq k,n}^a) \leq 1 - \frac{a!}{(a-n)!a^n}.\]
\item These two bounds imply that $Y_{\geq k,n}^a$ converges in distribution to $\mathcal{L}_n$ as $a \to \infty$ and that $Y_{\geq k,n}^p$ converges in distribution to $\mathcal{L}_{\geq k,n}$ as $a \to \infty$.
\item For a given $a$, among all distributions on $[a]$, the uniform distribution maximizes the expected number of nodes  with at least $k$ descendants. That is, if $p$ is any distribution on $[a]$, then $\E[Y_{\geq k,n}^p] \leq \E[Y_{\geq k,n}^a]$.
\item Same results also hold for the   number of nodes with exactly  $k$ descendants in a URT, $p$-BRT and $a$-RT of order $n$  which we denote by  $\mathcal{L}_{k,n}$, $Y_{k,n}^p$ and $Y_{k,n}^{a}$, respectively.
\end{enumerate}
\end{theorem}
Note that the results in Theorem \ref{thm:compare} can also be derived  for the number of branches and  the depth of node  $n$ in a similar way.  

We may now state the main result for the  URT case. 

\begin{theorem}
 Let $\mathcal{L}_{\geq k, n}$ be the number of nodes with at least $k$ descendants in a URT. Then we have $$\frac{\mathcal{L}_{\geq k, n} - \frac{n}{k+1}}{\sqrt{n \left(\frac{1}{k+1} + \frac{2H_{2k+1}-2H_{k+1}}{k+1} - \frac{2k+1}{(k+1)^2} \right)}} \longrightarrow_d \mathcal{G}, $$ as $n \rightarrow \infty$.
\end{theorem}
\begin{proof} Since $Y_{\geq k, n}^a \rightarrow_d \mathcal{L}_{k,n}$, as $a \rightarrow \infty$ we have $$\mathbb{E}[Y_{\geq k, n}^a] \longrightarrow \mathbb{E}[\mathcal{L}_{k,n}],$$ and  $$Var(Y_{\geq k, n}^a) \longrightarrow Var(\mathcal{L}_{k,n}),$$ as $a \rightarrow \infty$. This can be rigorously justified by using 	Theorem 5.9 of \cite{Gut} since, $\{Y_{\geq k, n}^a\}_{a \geq 1}$ is clearly uniformly integrable.  
We write 
 $$\frac{\mathcal{L}_{\geq k,n} - \mathbb{E}[\mathcal{L}_{\geq k,n} ]}{\sqrt{\Var (\mathcal{L}_{\geq k,n} )}}  = \frac{\mathcal{L}_{\geq k,n} - \mathcal{L}_{\geq k,n}^{a}}{\sqrt{\Var (\mathcal{L}_{\geq k,n} )}}
 +  \frac{\mathcal{L}_{\geq k,n}^a - \mathbb{E}[\mathcal{L}_{\geq k,n}^a ]}{\sqrt{\Var (\mathcal{L}_{\geq k,n}^a )}}  \frac{\sqrt{\Var (\mathcal{L}_{\geq k,n}^a )}}{\sqrt{\Var (\mathcal{L}_{\geq k,n})}}
+  \frac{\mathbb{E}[\mathcal{L}_{\geq k,n}^a] - \mathbb{E}[\mathcal{L}_{\geq k,n} ]}{\sqrt{\Var (\mathcal{L}_{\geq k,n} )}}$$
We now take the limit of both sides as $a, n \rightarrow \infty$ with $a = 2n$. Then the first and the third terms on the right-hand side converge to 0 a.s., and the second term converges in distribution to a standard normal. Result follows. 
\hfill $\square$
\end{proof}

We may use the approach above to understand the nodes with exactly $k$ descendants in a URT. This was previously studied in \cite{Devroye}. 

\begin{theorem} \cite{Devroye}
Let $\mathcal{L}_{k, n}$ be the number of nodes with exactly $k$ descendants in a URT.
$$\frac{\mathcal{L}_{k, n} - \frac{n}{(k+1)(k+2)}}{\sqrt{\Var(\mathcal{L}_{k, n})}} \longrightarrow_d \mathcal{G}, $$ as $n \rightarrow \infty$. Here, $\Var(\mathcal{L}_{k, n}) = \lim_{a \rightarrow \infty} \Var(X_{k,n}^{a})$, where $ \Var(X_{k,n}^{a})$ is as in \eqref{eqn:varxkna}. 
\end{theorem}

\section{Depth of node $n$}\label{sec:depth}

For uniform recursive trees the depth of node $n$ can be computed by referring to the recursive construction steps see for example \cite{Devroye88} or \cite{Feng05}. Since we do not have such a construction for BRTs, we will need to proceed differently here.
First we observe that given a permutation representation for a recursive tree we can determine $\mathcal{D}_n^{p}$, the depth of node $n$, by counting the number of anti-records when we go to the left  from the position of $n$ until 1. For example for $\gamma= 12574863$ we can see that the depth of node $8$ is 3, which corresponds to the number of anti-records when we go to the left from 8, which are 4, 2 and 1.

Thus given $\mathcal{P}_n = \gamma^{-1}(n)$, the position of $n$ in the permutation, we get the depth of node $n$ by calculating for all $i< \mathcal{P}_n$, the probability that $\gamma(i) < \min_{i<j<\mathcal{P}_n} \{\gamma(j)\}$. We will now use these observations to get the expected value and variance of the depth of $n$ in a BRT. We will need the following lemma:

\begin{lemma}
Let $\mathcal{P}_{n}^{p}$ denote the position of $n$ in a $p$-biased random permutation. Then for ${2 \leq k < n}$,
\begin{equation}
\Pro \left ( \mathcal{P}_{n}^{p} = k\right )  =
 \sum_{s=2}^{a} p_s\left (\sum_{r=1}^{s} p_r \right ) ^{k-2} \left  (\sum_{r=1}^{s-1} p_r \right ) ^{n-k}
\end{equation}
and
\begin{equation}
\Pro \left( \mathcal{P}_{n}^{p}= n \right) =  \sum_{s=1}^{a} p_s\left (\sum_{r=1}^{s} p_r \right ) ^{n-2}.
\end{equation}
\end{lemma}
\begin{proof}
As before we will use the construction of an inverse biased riffle shuffle permutation. We know that $n$ will be the last card in the last non-empty pile. Moreover the indices that get digit $s$, are the positions of the cards in the $s$-th pile. This means that the position of $n$ is the last index that gets the highest digit. We get
\begin{equation} \mathcal{P}_{n}^{p} = \max \left \{2 \leq k \leq n: X_k \geq \{X_2, \dots, X_n\} \right \}
\end{equation}
so 
\begin{equation}\Pro(  \mathcal{P}_{n}^{p} = k) = \Pro \left ( X_k \geq \{X_2, \dots, X_{k-1} \}, X_k > \{X_{k+1}, \dots, X_{n}\} \right ).
\end{equation}

By conditioning on $X_k$ we get independent events and can thus calculate this probability for $2\leq k<n$. We get
\begin{equation}
\begin{split}
\Pro& (  \mathcal{P}_{n}^{p} = k) \\
&= \Pro \left ( X_k \geq \{X_2, \dots, X_{k-1} \}, X_k > \{X_{k+1}, \dots, X_{n}\} \right ) \\
&= \sum_{s=1}^{a} \Pro \left ( X_k \geq \{X_2, \dots, X_{k-1} \}, X_k > \{X_{k+1}, \dots, X_{n}\} | X_k =s \right ) \Pro(X_k =s)\\
&= \sum_{s=1}^{a} \Pro \left ( s \geq \{X_2, \dots, X_{k-1} \}, s > \{X_{k+1}, \dots, X_{n}\} \right ) \Pro(X_k =s)\\
&= \sum_{s=1}^{a} \Pro \left ( s \geq \{X_2, \dots, X_{k-1} \}\right ) \Pro \left ( s > \{X_{k+1}, \dots, X_{n}\} \right ) \Pro(X_k =s)\\
&= \sum_{s=2}^{a} p_s\left (\sum_{r=1}^{s} p_r \right ) ^{k-2} \left (\sum_{r=1}^{s-1} p_r \right ) ^{n-k}. \\
\end{split}
\end{equation}
We only sum from $s=2$ since if $X_k=1$, $X_k$ cannot be strictly greater than $X_{k+1}, \dots, X_n$. However, if $k=n$, we have the additional possibility that all $X_i$ are equal to 1 and thus get 
\begin{equation}\Pro(  \mathcal{P}_{n}^{p} = n) =  \sum_{s=1}^{a} p_s\left (\sum_{r=1}^{s} p_r \right ) ^{n-2}.
\end{equation}
\hfill $\square$
\end{proof}

\begin{theorem} \label{ThmDepthBRT}
Let $\mathcal{D}_n^{p}$ denote the depth of node $n$ in a $p$-BRT. Then 
\begin{eqnarray*} \label{eq:DepthBRT}
\E[\mathcal{D}_n^{p}] &=& \sum_{s=2}^{a} \frac{p_s}{\sum_{r=1}^{s-1} p_r} \sum_{s'=2}^{a} \left [\left (\sum_{r=1}^{s'} p_r \right )^{n-1}  - \left ( \sum_{r=1}^{s'-1} p_r \right )^{n-1} \right ] \\
&& -\sum_{s=2}^{a} \frac{p_s}{\sum_{r=1}^{s-1} p_r} \sum_{s'=2}^{a} p_{s'} \frac{\left (\sum_{r=1}^{s'-1} p_r \right ) ^{n-1}  - \left ( \sum_{r=s}^{a} p_r  \sum_{r=1}^{s'} p_r \right )^{n-1}}{\sum_{r=1}^{s'-1} p_r - \sum_{r=s}^{a} p_r  \sum_{r=1}^{s'} p_r  }
 + p_1  \sum_{s=2}^{a} \frac{1}{p_s} \\
&& \qquad  \cdot \left [(n-2)\left (\sum_{r=1}^{s} p_r \right)^{n}- (n-1)\left (\sum_{r=1}^{s} p_r\right)^{n-1}\sum_{r=1}^{s-1} p_r  + \sum_{r=1}^{s} p_r \left (\sum_{r=1}^{s-1} p_r \right )^{n-1} \right ] \\
&&  + \sum_{s=2}^{a} \left( \left (\sum_{r=1}^{s} p_r \right ) ^{n-1} -  \left (\sum_{r=1}^{s-1} p_r \right ) ^{n-1}  \right ) + \left [\sum_{s=2}^{a} p_s \frac{1 - \left ( \sum_{r=s}^{a} p_r \right )^{n-2}}{\sum_{r=1}^{s-1} p_r}+(n-2)p_1 +1\right]p_1^{n-1}. 
\end{eqnarray*}
Moreover asymptotically we have
\begin{equation} \label{equ:DepthBRTninf}
\lim_{n\to \infty} \frac{\E[\mathcal{D}_n^{p}]}{n} = p_1.
\end{equation}
\end{theorem}

\begin{proof}
First of all we will find an expression for $\E[\mathcal{D}_n^{p} |  \mathcal{P}_{n}^{p} =k]$. As we said above, given $ \mathcal{P}_{n}^{p}$, we need to calculate the number of anti-records when we go from the position of $n$ to the left until we reach 1. Hence we define for $1 \leq i <  \mathcal{P}_{n}^{p}$, 
\begin{equation}R_i := \1 \left (\gamma(i) = \min_{i\leq k\leq  \mathcal{P}_{n}^{p}} \{\gamma(k)\}\right ).\end{equation} 
Then, given $ \mathcal{P}_{n}^{p}$, we get $\mathcal{D}_n^{p} = \sum_{i=1}^{ \mathcal{P}_{n}^{p}-1} R_i$, and so
\begin{equation} \E[\mathcal{D}_n^{p} |  \mathcal{P}_{n}^{p}] = \sum_{i=1}^{ \mathcal{P}_{n}^{p}-1} \E[R_i |  \mathcal{P}_{n}^{p}].\end{equation}
     We can simplify this sum by first observing that 
\begin{equation}\E[R_{ \mathcal{P}_{n}^{p}-1}] = \Pro( \gamma({ \mathcal{P}_{n}^{p}-1}) < \gamma({ \mathcal{P}_{n}^{p}})) = \Pro(\gamma({ \mathcal{P}_{n}^{p}-1}) < n) =1
\end{equation} and in general for all $i$, we have $\gamma(i) < \gamma( \mathcal{P}_{n}^{p}) = n$, so we can rewrite $R_i$ as \begin{equation}R_i = \1 \left (\gamma(i) = \min_{i\leq k<  \mathcal{P}_{n}^{p}} \{\gamma(k)\}\right ).\end{equation} 
Moreover $\Pro(R_1) = \Pro( \gamma(1) = \min_{1\leq k \leq  \mathcal{P}_{n}^{p}}\{\gamma(k)\} ) =1$ since $\gamma(1)=1$.
Now we will again use the inverse riffle shuffle construction to calculate the rest of these probabilities. We know that for $i<j$, $\gamma(i) < \gamma(j)$ if and only if $X_i \leq X_j$, since this means that in the spot $j$ will come a higher card from the same pile or from a pile corresponding to a higher digit, thus with higher labeled cards. 

Let $2\leq i \leq  \mathcal{P}_{n}^{p}$, then 
\begin{equation*}
\begin{split}
\E&[R_i |  \mathcal{P}_{n}^{p}] = \Pro \left ( \gamma(i) = \min_{i \leq k<  \mathcal{P}_{n}^{p}}  \{\gamma(k)\} | \gamma( \mathcal{P}_{n}^{p}) = n \right)  \\
&= \Pro \left ( X_i = \min_{i \leq k<  \mathcal{P}_{n}^{p}}  \{X_k\} | X_{ \mathcal{P}_{n}^{p}} \geq \{X_2, \dots, X_{ \mathcal{P}_{n}^{p}-1} \}, X_{ \mathcal{P}_{n}^{p}} > \{X_{ \mathcal{P}_{n}^{p}+1}, \dots, X_{n} \} \right ) 
= \Pro \left ( X_i = \min_{i \leq k<  \mathcal{P}_{n}^{p}}  \{X_k\}\right ) \\
&= \sum_{s=1}^{a}\Pro \left ( X_i = \min_{i \leq k<  \mathcal{P}_{n}^{p}}  \{X_k\} | X_i = s  \right ) \Pro ( X_i =s) = \sum_{s=1}^{a}\Pro \left ( s \leq \{X_{i+1}, \dots,  X_{ \mathcal{P}_{n}^{p}-1} \} \right ) \Pro ( X_i =s) \\
&= \sum_{s=1}^{a} p_s \left ( \sum_{r=s}^{a} p_r \right )^{ \mathcal{P}_{n}^{p}-i-1}. 
\end{split}
\end{equation*}
We could get rid of the conditional in the third line because the $X_k$ are mutually independent and thus $X_{ \mathcal{P}_{n}^{p}} \geq \{X_2, \dots, X_{ \mathcal{P}_{n}^{p}-1} \}$ is independent from the ordering of $ \{X_2, \dots, X_{ \mathcal{P}_{n}^{p}-1} \}$. Moreover $X_{ \mathcal{P}_{n}^{p}} > \{X_{ \mathcal{P}_{n}^{p}+1}, \dots, X_{n} \}$ concerns different $X_i$'s, so is also independent of $R_i$.
In total we thus get, for $2 < k \leq n$,
\begin{equation*}
\E[\mathcal{D}_n^{p} | \mathcal{P}_{n}^{p}=k] = \sum_{i=1}^{k-1} \E[R_i | \mathcal{P}_{n}^{p} = k] 
= \sum_{s=1}^{a} p_s \sum_{i=0}^{k-3}  \left ( \sum_{r=s}^{a} p_r \right )^{i} +1
= \sum_{s=2}^{a} p_s \frac{1 - \left ( \sum_{r=s}^{a} p_r \right )^{k-2}}{\sum_{r=1}^{s-1} p_r}+(k-2)p_1 +1\
\end{equation*}
and \begin{equation*}\E[\mathcal{D}_n^{p}| \mathcal{P}_{n}^{p} = 2] = 1.
\end{equation*}
Using the tower rule we then get after some simplifications an expression for the expectation and as an immediate consequence the asymptotic result, for details see \cite{ellatez}.
\hfill $\square$
\end{proof}

As a corollary when $p$ is the uniform distribution on  $[a]$ we obtain the following result whose proof is elementary but lenghty \cite{ellatez}. 

\begin{corollary} \label{cor:aRTExptDepth}
Let $\mathcal{D}_n^{a}$ denote the depth of node $n$ in an $a$-RT. Then 
\begin{equation*}
\begin{split}
\E\left [\mathcal{D}_n^{a}\right ] = &  H_{a-1} +1  + \frac{n-2}{a^{n}} - \frac{1}{a^{2n-3}}\sum_{s=1}^{a-1} \frac{(a-s)^{n-2}}{ s}   - \frac{1}{a^{n-2}} \sum_{s=1}^{a-1}  \sum_{s'=1}^{a-1} \frac{1}{s} \frac{{s' }^{n-1}  - \left( (a-s)(s'+1)\right )^{n-1}}{ss' +s -a }\\
& + \frac{1}{a^n} \sum_{s=1}^{a-1}
(n-2)\left (s+1\right)^{n}- (n-1)\left ( s+1\right)^{n-1}s  + (s+1) s^{n-1}.\\
\end{split}
\end{equation*}
Moreover asymptotically
\begin{equation*}
\frac{\E\left [\mathcal{D}_n^{a} \right ]}{n} \xrightarrow {n \to \infty} \frac{1}{a}
\end{equation*}
and as $a$ approaches infinity, we get the same expectation as for URTs:
\begin{equation*}
\E\left [\mathcal{D}_n^{a} \right ] \xrightarrow{a \to \infty} H_{n-1} .
\end{equation*}
\end{corollary}

\begin{remark}
There are various results in the literature on central limit theorems for the number of records in random words, for example \cite{Gouet05} and \cite{Bai98}. However in both of these papers the potential number of records is assumed to be infinite, which is not the case here. We hope to adapt these for obtaining asymptotic results on the number of branches in a subsequent work. 
\end{remark}

\section{Conclusion} \label{sec:conclusion}

We saw that in a biased recursive tree constructed from a riffle shuffle permutation based on the cutting of the deck into $a$ piles, the maximum degree is $a$. Thus an interesting question would be to compare $a$-ary recursive tree with biased recursive trees and especially $a$-recursive trees. It is not clear if this common restriction on the degree of the nodes is the only property these trees have in common or if they are more similar that one might expect at first sight. In order to get insight concerning this relation a dynamic construction of biased recursive trees would be very useful. If such a construction exists it would probably be very different from the construction principles we know. Of course such a dynamic growth rule would be useful for many other questions as well.

Throughout the study of biased recursive trees, the use of riffle shuffles instead of uniformly random permutations stemmed from the fact that certain statistics could be expressed in terms of independent random variables. As mentioned above, there is another random permutation framework allowing such use of independence, namely, unfair permutations. See \cite{Prodinger} and \cite{AIP} for the definition and analysis for various statistics of unfair permutations. Replacing  riffle shuffles with unfair permutations can make a big difference   since the case equality in random words case disappears in unfair permutation setting, and we believe that this should be checked in a subsequent work. Since in that model the rank of $i$ is determined by the maximum of $i$ identically distributed independent uniform random variables, it probably has similarities to models where each node can choose among $k$ potential parents, see \cite{DSouza07}, \cite{Mahmoud10}.

As a last note, our motivation for studying BRTs was to be able to understand a general class of random trees with arbitrary connection probabilities. The idea is to show that the biased recursive trees  are dense in  a given class of trees and to then approximate the statistics of this non-uniform random tree model with BRT statistics, and then use the underlying independence. However, we were not able to make an important progress in this direction yet. It is also kept for future research. 

\bigskip 

\noindent \textbf{Acknowledgements:}  The second author is supported partially by TUBITAK BIDEB 2232.

\appendix

\section{Variance computation in  Theorem \ref{thm:kDesVarGeneral}}
 
 We will now calculate the variance of $Y_{\geq k, n}^{p}$, the number of nodes with at least $k$ descendants in a BRT. We will use the following expression for the variance:
\begin{eqnarray*}
\Var(Y_{\geq k, n}^{p}) &=& \Var \left( \sum_{i=2}^{n-k} C_{i}^{k} \right) 
= \E\left[\left (\sum_{i=2}^{n-k} C_{i}^{k} \right )^2\right] - \E\left [\sum_{i=2}^{n-k} C_{i}^{k} \right]^2 
=  \sum_{i=2}^{n-k}  \E\left[{C_{i}^{k}}^2 \right]- \sum_{i=2}^{n-k}  \E\left [C_{i}^{k} \right]^2\\
&& + 2 \sum_{i=2}^{n-k-1} \sum_{j=i+1}^{n-k} \E \left [ C_{i}^{k} C_{j}^{k} \right] - 2 \sum_{i=2}^{n-k-1} \sum_{j=i+1}^{n-k} \E \left [ C_{i}^{k} \right] \E \left [ C_{j}^{k} \right].  
\end{eqnarray*}
First of all we have
$$
\sum_{i=2}^{n-k} \E\left [{C_i^k}^2\right ] = \sum_{i=2}^{n-k} \E\left [{C_i^k}\right ] = (n-k-1)\sum_{s=1}^{a} p_s \left (\sum_{r=s}^{a} p_r  \right )^k, 
$$
and
$$\sum_{i=2}^{n-k} \E \left [C_i^k \right ]^2 = (n-k-1) \left (\sum_{s=1}^{a} p_s \left (\sum_{r=s}^{a} p_r  \right )^k \right )^2. $$
Since the events $C_i^k$ and $C_j^k$ are not mutually independent when $i< j \leq i+k$, we also need to express $\E[C_i^kC_j^k]$ for $i < j \leq i+k$. 

Let $2 \leq i < j \leq i+k$, then
\begin{equation*}
\begin{split}
&\E[C_i^k C_j^k] \\
&= \Pro\left (X_i \leq \{X_{i+1}, \dots, X_{i+k}\},  X_j \leq \{X_{j+1}, \dots, X_{j+k}\}\right )\\
%&= \Pro\left (X_i \leq \{X_{i+1}, \dots, X_{j-1}\},  X_i \leq X_j, X_j \leq \{X_{j+1}, \dots, X_{j+k}\}\right )\\
&= \sum_{s=1}^{a} \Pro\left (X_i \leq \{X_{i+1}, \dots, X_{j-1}\}, X_i \leq X_j, X_j \leq \{X_{j+1}, \dots, X_{j+k}\} | X_i =s\right )\Pro(X_i=s) \\
&= \sum_{s=1}^{a} \Pro\left (X_i \leq \{X_{i+1}, \dots, X_{j-1}\} | X_i \leq X_j, X_j \leq \{X_{j+1}, \dots, X_{j+k}\}, X_i =s\right )  \\
&\hspace{4em} \Pro\left (X_j \leq \{X_{j+1}, \dots, X_{j+k}\} | X_i \leq X_j, X_i =s\right ) \Pro\left (X_i \leq X_j | X_i =s\right ) \Pro(X_i=s) \\
&= \sum_{s=1}^{a} \Pro(X_i=s) \Pro\left (s \leq \{X_{i+1}, \dots, X_{j-1}\} \right) \Pro\left (s \leq X_j\right ) \Pro\left (X_j \leq \{X_{j+1}, \dots, X_{j+k}\} | s \leq X_j \right)  \\
&= \sum_{s=1}^{a} \Pro(X_i=s) \Pro\left (s \leq \{X_{i+1}, \dots, X_{j-1}\} \right) \sum_{r=s}^{a} \Pro\left (X_j = r \right ) \Pro\left (X_j \leq \{X_{j+1}, \dots, X_{j+k}\} | X_j = r\right)  \\
&= \sum_{s=1}^{a} \Pro(X_i=s) \Pro\left (s \leq \{X_{i+1}, \dots, X_{j-1}\} \right) \sum_{r=s}^{a} \Pro\left (X_j = r \right ) \Pro\left (r \leq \{X_{j+1}, \dots, X_{j+k}\}\right )  \\
&= \sum_{s=1}^{a} p_s  \left ( \sum_{u=s}^{a} p_u\right )^{j-i-1}\sum_{r=s}^{a} p_r \left ( \sum_{t=r}^{a} p_t \right )^k. \\
\end{split}
\end{equation*}

Now we have
\begin{equation*}
\begin{split}
\sum_{i=2}^{n-k-1} & \sum_{j =  i+1}^{n-k} \E[C_i^k C_j^k] - \sum_{i=2}^{n-k-1} \sum_{j =  i+1}^{n-k} \E[C_i^k]\E[C_j^k] \\
&= \sum_{i=2}^{n-2k} \sum_{j =  i+1}^{i+k} \E[C_i^k C_j^k]  - \sum_{i=2}^{n-2k} \sum_{j =  i+1}^{i+k}\E[C_i^k]\E[C_j^k] \\
& \quad + \sum_{i=n-2k+1}^{n-k-1} \sum_{j =  i+1}^{n-k} \E[C_i^k C_j^k]  - \sum_{i=n-2k+1}^{n-k-1} \sum_{j =  i+1}^{n-k} \E[C_i^k]\E[C_j^k],
\end{split}
\end{equation*}
where we eliminated all terms where $i+k<j$ and thus $\E[C_i^kC_j^k] - \E[C_i^k]\E[C_j^k]=0$. Moreover we separated the two sums because if $i+k > n-k$, only the $C_j^k$ until $n-k$, not until $i+k$ are relevant. We now consider all four expressions separately. We get
\begin{equation*}
\begin{split}
\sum_{i=2}^{n-2k}  \sum_{j =  i+1}^{i+k} \E[C_i^k C_j^k] &= \sum_{i=2}^{n-2k} \sum_{j =  i+1}^{i+k} \sum_{s=1}^{a} p_s  \left ( \sum_{u=s}^{a} p_u\right )^{j-i-1}\sum_{r=s}^{a} p_r \left ( \sum_{t=r}^{a} p_t \right )^k \\
&=  \sum_{i=2}^{n-2k} p_1  \sum_{j =  0}^{k-1}  \left ( \sum_{u=1}^{a} p_u\right )^{j}\sum_{r=1}^{a} p_r \left ( \sum_{t=r}^{a} p_t \right )^k \\
 &\hspace{2em} + \sum_{i=2}^{n-2k} \sum_{s=2}^{a} p_s  \sum_{j =  0}^{k-1}  \left ( \sum_{u=s}^{a} p_u\right )^{j}\sum_{r=s}^{a} p_r \left ( \sum_{t=r}^{a} p_t \right )^k \\
&=  (n-2k-1) p_1  k \sum_{r=1}^{a} p_r \left ( \sum_{t=r}^{a} p_t \right )^k \\
 &\hspace{2em} + (n-2k-1) \sum_{s=2}^{a} p_s  \frac{1-\left (\sum_{u=s}^{a} p_u\right )^k}{\sum_{u=1}^{s-1} p_u}
\sum_{r=s}^{a} p_r \left ( \sum_{t=r}^{a} p_t \right )^k .\\
\end{split}
\end{equation*}
Also
\begin{equation*}
\begin{split}
\sum_{i=2}^{n-2k} \sum_{j =  i+1}^{i+k} \E[C_i^k]\E[C_j^k]
&= \sum_{i=2}^{n-2k} \sum_{j =  i+1}^{i+k}  \left (\sum_{s=1}^{a} p_s \left (\sum_{r=s}^{a} p_r  \right )^k \right ) ^2 \\
& = (n-2k-1) k  \left (\sum_{s=1}^{a} p_s \left (\sum_{r=s}^{a} p_r  \right )^k \right ) ^2.\\
\end{split}
\end{equation*}
For the second part of the sum we get
\begin{equation*}
\begin{split}
 &\sum_{i=n-2k+1}^{n-k-1} \sum_{j = i+1}^{n-k} \E[C_i^k C_j^k] \\
&=  \sum_{i=n-2k+1}^{n-k-1} \sum_{j = i+1}^{n-k} \sum_{s=1}^{a} p_s  \left ( \sum_{u=s}^{a} p_u\right )^{j-i-1}\sum_{r=s}^{a} p_r \left ( \sum_{t=r}^{a} p_t \right )^k \\
&= \sum_{i=n-2k+1}^{n-k-1} \sum_{j = 0}^{n-k-i-1} p_1  \left ( \sum_{u=1}^{a} p_u\right )^{j} \sum_{r=1}^{a} p_r \left ( \sum_{t=r}^{a} p_t \right )^k\\
&\hspace{2em} + \sum_{i=n-2k+1}^{n-k-1} \sum_{s=2}^{a} p_s \frac{ 1- \left ( \sum_{u=s}^{a} p_u\right )^{n-k-i}}{1- \sum_{u=s}^{a} p_u} \sum_{r=s}^{a} p_r \left ( \sum_{t=r}^{a} p_t \right )^k  \\
&= \sum_{i=n-2k+1}^{n-k-1} (n-k-i) p_1 \sum_{r=1}^{a} p_r \left ( \sum_{t=r}^{a} p_t \right )^k\\
&\hspace{2em} +  \sum_{i=n-2k+1}^{n-k-1} \sum_{s=2}^{a} p_s \frac{ 1}{1- \sum_{u=s}^{a} p_u} \sum_{r=s}^{a} p_r \left ( \sum_{t=r}^{a} p_t \right )^k  \\
&\hspace{2em} - \sum_{i=n-2k+1}^{n-k-1} \sum_{s=2}^{a} p_s \frac{\left ( \sum_{u=s}^{a} p_u\right )^{n-k-i}}{1- \sum_{u=s}^{a} p_u} \sum_{r=s}^{a} p_r \left ( \sum_{t=r}^{a} p_t \right )^k\\
&= \frac{k(k-1)}{2} p_1 \sum_{r=1}^{a} p_r \left ( \sum_{t=r}^{a} p_t \right )^k\\ 
&\hspace{2em} + (k-1)\sum_{s=2}^{a} p_s \frac{ 1}{ \sum_{u=1}^{s-1} p_u} \sum_{r=s}^{a} p_r \left ( \sum_{t=r}^{a} p_t \right )^k \\
&\hspace{2em}  -  \sum_{s=2}^{a} p_s \frac{ 1}{ \sum_{u=1}^{s-1} p_u} \sum_{i=1}^{k-1} \left ( \sum_{u=s}^{a} p_u\right )^{i}  \sum_{r=s}^{a} p_r \left ( \sum_{t=r}^{a} p_t \right )^k\\
&\hspace{2em}  -  \sum_{s=2}^{a} p_s \frac{ 1}{ \sum_{u=1}^{s-1} p_u}\left (  \frac{1- \left ( \sum_{u=s}^{a} p_u\right )^{k}}{1-\sum_{u=s}^{a} p_u} -1 \right ) \sum_{r=s}^{a} p_r \left ( \sum_{t=r}^{a} p_t \right )^k.\\
\end{split}
\end{equation*} 
Moreover we have
\begin{equation*}
\begin{split}
 \sum_{i=n-2k+1}^{n-k-1} \sum_{j = i+1}^{n-k}  \E[C_i^k]\E[C_j^k]
 &=  \sum_{i=n-2k+1}^{n-k-1} \sum_{j = i+1}^{n-k} \left ( \sum_{s=1}^{a} p_s \left (\sum_{r=s}^{a} p_r  \right )^k \right )^2 \\
  &=  \frac{k(k-1)}{2}\left ( \sum_{s=1}^{a} p_s \left (\sum_{r=s}^{a} p_r  \right )^k \right )^2 . \\
\end{split}
\end{equation*}

After adding up these terms, we get  
\begin{equation*}
\begin{split}
\Var(Y_n^k) &= \sum_{i=2}^{n-k} \E[{C_i^k}^2] + 2 \sum_{i=2}^{n-2k} \sum_{j =  i+1}^{i+k} \E[C_i^k C_j^k] + 2 \sum_{i=n-2k+1}^{n-k-1} \sum_{j = i+1}^{n-k} \E[C_i^k C_j^k]\\
& \hspace{1ex} - \sum_{i=2}^{n-k} \E[C_i^k]^2 -  2 \sum_{i=2}^{n-2k} \sum_{j =  i+1}^{i+k} \E[C_i^k]\E[C_j^k] - 2 \sum_{i=n-2k+1}^{n-k-1} \sum_{j = i+1}^{n-k}  \E[C_i^k]\E[C_j^k]  \\
&=  (n-k-1)\sum_{s=1}^{a} p_s \left (\sum_{r=s}^{a} p_r  \right )^k \\
&\hspace{2em} + 2 (n-2k-1) p_1  k \sum_{r=1}^{a} p_r \left ( \sum_{t=r}^{a} p_t \right )^k \\
 &\hspace{2em} + 2 (n-2k-1) \sum_{s=2}^{a} p_s  \frac{1-\left (\sum_{u=s}^{a} p_u\right )^k}{\sum_{u=1}^{s-1} p_u}
\sum_{r=s}^{a} p_r \left ( \sum_{t=r}^{a} p_t \right )^k \\
&\hspace{2em} + 2 \frac{k(k-1)}{2} p_1 \sum_{r=1}^{a} p_r \left ( \sum_{t=r}^{a} p_t \right )^k\\ 
&\hspace{2em} +  2 (k-1)\sum_{s=2}^{a} p_s \frac{ 1}{ \sum_{u=1}^{s-1} p_u} \sum_{r=s}^{a} p_r \left ( \sum_{t=r}^{a} p_t \right )^k \\
&\hspace{2em}  -  2 \sum_{s=2}^{a} p_s \frac{ 1}{ \sum_{u=1}^{s-1} p_u}\left (  \frac{1- \left ( \sum_{u=s}^{a} p_u\right )^{k}}{1-\sum_{u=s}^{a} p_u} -1 \right ) \sum_{r=s}^{a} p_r \left ( \sum_{t=r}^{a} p_t \right )^k \\
& \hspace{2em} - (n-k-1) \left (\sum_{s=1}^{a} p_s \left (\sum_{r=s}^{a} p_r  \right )^k \right )^2 \\
&\hspace{2em} - 2 (n-2k-1) k  \left (\sum_{s=1}^{a} p_s \left (\sum_{r=s}^{a} p_r  \right )^k \right ) ^2 \\
&\hspace{2em} - 2 \frac{k(k-1)}{2}\left ( \sum_{s=1}^{a} p_s \left (\sum_{r=s}^{a} p_r  \right )^k \right )^2
\end{split}
\end{equation*}
After some grouping some terms and some simplifications this finally gives
\begin{equation*}
\begin{split}
\Var(Y_{\geq k, n}^p)= &\sum_{s=1}^{a} p_s \left (\sum_{r=s}^{a} p_r  \right )^k \left [(n-k-1)+ p_1\left(2nk -3k(k+1)\right) \right ] \\
&+2 \sum_{s=2}^{a} p_s  \frac{1}{\sum_{u=1}^{s-1} p_u} \sum_{r=s}^{a} p_r \left ( \sum_{t=r}^{a} p_t \right )^k \\
& \hspace{1em} \cdot \left [n-k-1  -(n-2k-1)\left (\sum_{u=s}^{a} p_u\right )^k - \frac{1- \left ( \sum_{u=s}^{a} p_u\right )^{k}}{1-\sum_{u=s}^{a} p_u}  \right ]\\
& -  \left ( \sum_{s=1}^{a} p_s \left (\sum_{r=s}^{a} p_r  \right )^k \right )^2 \left [ n(2k+1) - (3k+1)(k+1)\right ].
\end{split}
\end{equation*}
This immediately implies for $k$ fixed as $n \to \infty$,
\begin{equation*}
\begin{split}
\lim_{n \to \infty} \frac{\Var(Y_{\geq k, n}^{p})}{n} & =  \sum_{s=1}^{a} p_s \left (\sum_{r=s}^{a} p_r  \right )^k \left ( 2kp_1 + 1 -  (2k+1) \sum_{s=1}^{a} p_s \left (\sum_{r=s}^{a} p_r  \right )^k \right )  \\
&\hspace{1em}+2 \sum_{s=2}^{a} \frac{p_s}{\sum_{u=1}^{s-1} p_u} \sum_{r=s}^{a} p_r \left ( \sum_{t=r}^{a} p_t \right )^k \left [1 -\left (\sum_{u=s}^{a} p_u\right )^k \right ].\\
\end{split}
\end{equation*}
\hfill $\square$
\section{Proof of Corollary \ref{cor:aRTkDescendantsVar}}

If we choose the uniform distribution over $[a]$, Theorem \ref{thm:kDesVarGeneral} gives
\begin{equation*}
\begin{split}
\Var(Y_{\geq k , n}^{a})&=  \sum_{s=1}^{a} \frac{1}{a} \left (\sum_{r=s}^{a} \frac{1}{a}  \right )^k \left [(n-k-1)+ \frac{1}{a}\left(2nk -3k(k+1)\right) \right ] \\
& \quad +2 \sum_{s=2}^{a} \frac{1}{a}  \frac{1}{\sum_{u=1}^{s-1} \frac{1}{a}} \sum_{r=s}^{a} \frac{1}{a} \left ( \sum_{t=r}^{a} \frac{1}{a} \right )^k \\
& \qquad \cdot \left [n-k-1  -(n-2k-1)\left (\sum_{u=s}^{a} \frac{1}{a}\right )^k - \frac{1- \left ( \sum_{u=s}^{a} \frac{1}{a}\right )^{k}}{1-\sum_{u=s}^{a} \frac{1}{a}}  \right ]\\
& \quad -  \left ( \sum_{s=1}^{a} \frac{1}{a} \left (\sum_{r=s}^{a} \frac{1}{a}  \right )^k \right )^2 \left [ n(2k+1) - (3k+1)(k+1) \right ].
\end{split}
\end{equation*}
and after some simplifications, in particular using
\begin{equation*}
 \sum_{s=1}^{a}  \left (\sum_{r=s}^{a} \frac{1}{a}  \right )^k = \sum_{s=1}^{a} \left (\frac{a-s+1}{a}  \right )^k = \frac{1}{a^k}\sum_{s=1}^{a} s^k
 \end{equation*}
and
\begin{equation*}
\begin{split}
\sum_{s=2}^{a} \frac{1}{a}  \frac{1}{\sum_{u=1}^{s-1} \frac{1}{a}} \sum_{r=s}^{a} \frac{1}{a} \left ( \sum_{t=r}^{a} \frac{1}{a} \right )^k & = \sum_{s=2}^{a} \frac{1}{s-1} \sum_{r=s}^{a} \frac{1}{a} \left (\frac{a-r+1}{a} \right )^k \\
&\hspace{1em} = \frac{1}{a^{k+1}} \sum_{s=2}^{a} \frac{1}{s-1} \sum_{r=1}^{a-s+1} r^k ,
\end{split}
\end{equation*}
we get
\begin{equation*}
\begin{split}
\Var(Y_{\geq k , n}^{a})= &\frac{1}{a^{k+1}}\sum_{s=1}^{a} s^k\left [(n-k-1)+ \frac{1}{a}\left(2nk -3k(k+1)\right) \right ] \\
&+2 \frac{1}{a^{k+1}} \sum_{s=2}^{a} \frac{1}{s-1} \sum_{r=1}^{a-s+1} r^k \\
& \hspace{1em} \cdot \left [n-k-1  -(n-2k-1)\left ( \frac{a-s+1}{a}\right )^k - \frac{1- \left ( \frac{a-s+1}{a}\right )^{k}}{1-\frac{a-s+1}{a}}  \right ]\\
& -  \left (\frac{1}{a^{k+1}}\sum_{s=1}^{a} s^k \right )^2 \left [ n(2k+1) - (3k+1)(k+1)\right ].
\end{split}
\end{equation*}
For $n \to \infty$ this directly gives, as it also follows from Theorem \ref{thm:kDesVarGeneral},
\begin{equation*}
\begin{split}
\lim_{n \to \infty} \frac{\Var(Y_{\geq k , n}^{a})}{n} =& \frac{1}{a^{k+1}}\sum_{s=1}^{a} s^k\left [1 + \frac{2k}{a} - \frac{2k+1}{a^{k+1}}\sum_{s=1}^{a} s^k \right ] \\
& + \frac{2}{a^{k+1}} \sum_{s=1}^{a-1} \frac{1}{s} \left [1  - \left ( \frac{a-s}{a}\right )^k  \right ] \sum_{r=1}^{a-s+1} r^k. \\
\end{split}
\end{equation*}
\hfill $\square$
 
 \end{document}